\documentclass[11pt]{article}
\usepackage[utf8]{inputenc}
\usepackage{lmodern, float}
\usepackage[margin=0.8in]{geometry}

\usepackage{tocloft}

\usepackage{booktabs}\usepackage[svgnames]{xcolor}
\usepackage[bookmarksnumbered=true]{hyperref} 
\hypersetup{
     colorlinks = true,
     linkcolor = blue,
     anchorcolor = blue,
     citecolor = teal,
     filecolor = blue,
     urlcolor = blue
     }

   \usepackage{amsmath}
   \usepackage{amssymb}
    
   \usepackage{bm}
\usepackage{soul}
\usepackage{amsmath}
\usepackage{amssymb}
\usepackage{amsthm}
\usepackage{wasysym}
\usepackage{mathtools}
\usepackage{graphicx}
\usepackage{hyperref}
\usepackage{booktabs}
\usepackage{bbold}
\usepackage{mathrsfs}
\usepackage{tikz-cd}
\usepackage{color}
\usepackage{setspace}
\usepackage{comment}
\usepackage{multirow}
\usepackage{bm}
\usepackage{thmtools}
\usepackage{thm-restate}
\usepackage{cleveref}
\usepackage{enumitem}
\usepackage{csquotes}
\usepackage{longtable}

\newcommand{\NN}{\mathbb{N}}
\newcommand{\ZZ}{\mathbb{Z}}

\newcommand{\RR}{\mathbb{R}}
\newcommand{\CC}{\mathbb{C}}
\newcommand{\PP}{\mathbb{P}}
\newcommand{\BBA}{\mathbb{A}}

\DeclareMathOperator{\rk}{rk}

\DeclareMathOperator{\Proj}{Proj}

\DeclareMathOperator{\Ima}{Im}
\DeclareMathOperator{\Trop}{Trop}

\DeclareMathOperator{\Cl}{Cl}

\newcommand{\MI}{\mathcal{I}}

\newcommand{\MP}{\mathcal{P}}

\newcommand{\Fl}{\mathscr{F}\!\ell}

\newcommand{\binomial}[2]{{ \mathbf{{#1}_{#2}}}}

\theoremstyle{definition}
\newtheorem{definition}{Definition}
\newtheorem{remark}{Remark}
\newtheorem{case}{Case}
\newtheorem{subcase}{Case}
\numberwithin{subcase}{case}
\newtheorem*{notation}{Notation}
\newtheorem{example}{Example}
\newtheorem{setup}{Setup}
\newtheorem{procedure}{Procedure}

\theoremstyle{plain}

\newtheorem{lemma}{Lemma}
\newtheorem{corollary}{Corollary}
\newtheorem{prop}{Proposition}
\newtheorem{conjecture}{Conjecture}

\newcommand{\Bsig}{B^\sigma}

\DeclareMathOperator{\argmin}{argmin}
\DeclareMathOperator{\Gr}{Gr}
\DeclareMathOperator{\conv}{Conv}
\DeclareMathOperator{\Ker}{ker}

\DeclareMathOperator{\sgn}{sgn}

\usepackage{newtxmath}      % provides definitions for: \jmath, \amalg, \coprod
\hypersetup{
     colorlinks = true,
     linkcolor = blue,
     anchorcolor = blue,
     citecolor = teal,
     filecolor = blue,
     urlcolor = black
     }

\usepackage{enumitem} % for '\newlist' and '\setlist' macros
\newlist{condenum}{enumerate}{1} % 'condenum': a new, enumerate-like list env.
\setlist[condenum]{label=\bfseries Condition \arabic*., 
                   ref=\arabic*, wide}

\title{Toric degenerations of partial flag varieties\\ and combinatorial mutations of matching field polytopes}
\author{Oliver Clarke, Fatemeh Mohammadi and Francesca Zaffalon
}
\begin{document}

\maketitle

\noindent{\bf Abstract.}
\noindent We study toric degenerations arising from Gr\"obner degenerations or the tropicalization of partial flag varieties. We produce a new family of toric degenerations of partial flag varieties whose combinatorics are governed by matching fields and combinatorial mutations of polytopes. We provide an explicit description of the polytopes associated with the resulting toric varieties in terms of matching field polytopes. These polytopes encode the combinatorial data of monomial degenerations of Pl\"ucker forms for the Grassmannian. We give a description of matching field polytopes of flag varieties as Minkowski sums and show that all such polytopes are normal. The polytopes we obtain are examples of Newton-Okounkov bodies for particular full-rank valuations for partial flag varieties. Furthermore, we study a certain explicitly-defined large family of matching field polytopes and prove that all polytopes in this family are connected by combinatorial mutations. Finally, we apply our methods to explicitly compute toric degenerations of small Grassmannians and flag varieties and obtain new families of toric degenerations.

{\hypersetup{linkcolor=black}
\setcounter{tocdepth}{2}\setlength\cftbeforesecskip{1.1pt}
{\tableofcontents}}

\section{Introduction}
\noindent{\bf Background and motivation.} We study toric degenerations of partial flag varieties, in particular Grassmannians and full flag varieties. As a set, the full flag variety $\Fl_n$ is the collection of complete flags $V_0 = \{0\}\subsetneq V_1 \subsetneq \cdots \subsetneq V_{n-1}\subsetneq V_n = \CC^n$ where $V_i$ is an $i$-dimensional vector subspace of $\CC^n$. More generally, for each subset ${K} \subseteq [n] := \{1, \dots, n\}$, the partial flag variety $\Fl({K};n)$ consists of the set of flags $V_{{k}_1} \subsetneq V_{{k}_2} \subsetneq \dots \subsetneq V_{{k}_d}$ where $V_{{k}_i}$ is a ${k}_i$-dimensional vector subspace of $\CC^n$, ${k}_i \in K$. For example, the Grassmannian $\Gr(k,n)$ is the partial flag variety $\Fl(k ; n)$ of $k$-dimensional linear subspaces of $\CC^n$. Grassmannians and (partial) flag varieties have been studied extensively in the literature from different perspectives such as algebraic geometry \cite{KOGAN, gonciulea1996degenerations},  representation theory \cite{makhlin2020gelfand, cerulli2020linear}, cluster algebras \cite{gross2014canonical, rietsch2019newton, bossinger2020toric}, as well as combinatorics \cite{brandt2021tropical}. In this paper, we will study their toric degenerations from a computational perspective.

A toric degeneration \cite{Anderson2013} of an algebraic variety $X$ is a flat family $\mathcal{F}\to \BBA^1$ whose fibers $\mathcal F_t$ over all points $t\in \BBA^{1}\setminus \{0\}$ are isomorphic to $X$ and whose fiber over $0$ is a toric variety $Y := \mathcal F_0$. Toric degenerations are particularly useful because many important algebraic invariants of $X$, such as the Hilbert polynomial and the degree, coincide with those of $Y$. Hence, it is often practical to study the invariants of $X$ using the toric variety $Y$ since toric varieties have rich combinatorial structure. This is due to a well-established dichotomy between toric varieties and discrete geometric objects, such as polyhedral fans and polytopes \cite{Cox2011}. We refer to a variety $Y$ as above as a toric degeneration of $X$.

\medskip

The study of toric degenerations is motivated by two natural questions: how do we obtain toric degenerations? And what is the relationship between two different toric degenerations of a fixed algebraic variety? To answer the first question, one general approach is through Gr\"obner degeneration \cite{Kaveh2019a}, which is further explained in \Cref{sec: toric degen from matching fields}. Such toric degenerations of a closed subvariety $V(\MI)$ of a large space are produced by toric (binomial and prime) initial ideals of $\MI$. The second question involves understanding properties of algebraic varieties that are preserved under a toric degeneration. 

\medskip

The classes of toric degenerations that we study arise from matching fields \cite{Sturmfels1993Maximal}. These have been introduced and studied for various kinds of homogeneous spaces, such as Grassmannians \cite{clarke2022combinatorial, clarke2022combinatorial2}, Schubert varieties \cite{Clarke2021b} and full flag varieties \cite{Clarke2021}. In the case of full flag varieties, the idea is as follows. Recall that the flag variety $\Fl_n$ naturally lives in a product of Grassmannians $\mathrm{Gr}(1,n)\times \dots \times \mathrm{Gr}(n-1,n)$ by identifying each flag $V_0 \subsetneq V_1 \subsetneq \cdots \subsetneq V_n$ with the tuple $(V_0, V_1, \dots, V_n)$. This product embeds, via the Pl\"ucker embedding, into a product of projective spaces $\PP^{\binom{n}{1} - 1}\times \dots \times \PP^{\binom{n}{n-1} - 1}$. On the level of rings, this embedding is obtained from the map 
\[
  \phi: 
  \CC\left[p_I \mid \varnothing \subsetneq I \subsetneq [n]\right] 
  \to \CC\left[x_{i,j} \mid i\in [n-1], \ j\in [n]\right], \ 
  p_I  \mapsto \det(X_I) ,
\]
where $X = (x_{i,j})$ is an $(n-1)\times n$ matrix of variables and $X 
_I$ is the submatrix of $X$ on the first $|I|$ rows {and on columns indexed by $I$,} for each $I \subset [n]$.
On the other hand, toric subvarieties of $\PP^{\binom{n}{1}-1}\times \dots \times \PP^{\binom{n}{n-1}-1}$ arise from monomial maps $\CC\left[p_I\ |\ \varnothing \subsetneq I \subsetneq [n]\right] \to S$, for some polynomial ring $S$.
Hence, a good candidate for a Gr\"obner degeneration of $\Fl_n$ is given by deforming $\phi$ to a monomial map. This is done by sending each \emph{Pl\"ucker variable} $p_I$ to one of the summands of $\det(X_I)$.
A \emph{matching field} is a combinatorial object which encodes this data. 

\medskip

Matching fields were originally introduced in \cite{Sturmfels1993Maximal} to study Newton polytopes of products of maximal minors. A matching field is a map $\Lambda$ that sends each variable $p_I$ to a permutation $\sigma \in S_{|I|}$. In particular, this gives a natural monomial map which corresponds to a toric subvariety of the product of projective spaces above. 
In order for a matching field $\Lambda$ to give rise to a toric degeneration of $\Fl_n$, it is necessary for $\Lambda$ to be \textit{coherent}, i.e., induced by a \textit{weight matrix} $M_\Lambda$. 
To obtain a toric degeneration, one takes the initial ideal of the Pl\"ucker ideal with respect to the image of $M_\Lambda$ under the \textit{tropical Stiefel map}~\cite{fink2015stiefel}. 

\medskip

Toric degenerations arising from matching fields have been studied in great detail in the literature \cite{clarke2020toric, Clarke2021, higashitani2022quadratic, Mohammadi2019}. Such degenerations correspond to top-dimensional cones of the tropical variety as in \cite{fink2015stiefel,Kaveh2019a}. There remain many open questions in this area. Notably, which matching fields produce toric degenerations?
In studying this question we use \textit{combinatorial mutations}, introduced in \cite{Akhtar2012}. A combinatorial mutation is a special kind of piecewise linear map between polytopes. We say that two polytopes are \textit{mutation equivalent} if there exists a sequence of combinatorial mutations between them.

\medskip

Each monomial map obtained from a coherent matching field $\Lambda$ gives rise to a projective toric variety $V(\ker(\phi_\Lambda))$. In general, this is not a toric degeneration of $\Fl_n$. We associate to $\Lambda$ the \textit{matching field polytope} $\mathcal{P}_\Lambda$ which is the toric polytope of $V(\ker(\phi_\Lambda))$. 
Note that the above constructions also apply to the case of partial flag varieties $\Fl({K};n)$. In particular, if $\Lambda$ is a matching field for $\Fl({K};n)$, then we write $\mathcal{P}_\Lambda^{{K}}$ for the corresponding matching field polytope. Whenever the partial flag variety is clear from context, we will omit it from the notation and simply write $\mathcal{P}_\Lambda$ for the polytope.

\medskip

\noindent{\bf Our contribution.} 
We study a large family of matching field polytopes and we show that: all such polytopes are normal; the property of giving rise to a toric degeneration is preserved by combinatorial mutation.
More precisely, a matching field $\Lambda$ inherits the property of giving rise to a toric degeneration from another matching field $\Lambda'$ whenever $\mathcal{P}_\Lambda$ and $\mathcal{P}_{\Lambda'}$ are mutation equivalent. 

\begin{restatable}{thm}{toricdegenpartialflag}\label{thm:toric degen partial flag}
Fix $n \in \ZZ$, $n>0$ and a subset ${K} \subseteq [n]$. Let $\Lambda$ be a matching field for the partial flag variety $\Fl({K};n)$. If the matching field polytope $\mathcal{P}_\Lambda$ is combinatorial mutation equivalent to the Gelfand-Tsetlin polytope, then $\Lambda$ gives rise to a toric degeneration of $\Fl({K};n)$. 
\end{restatable}
We also study a new family of coherent matching fields $B^\sigma$ for $\Fl({K};n)$ indexed by permutations $\sigma\in S_n$, which generalises many of the families studied in previous works \cite{clarke2022combinatorial, clarke2020toric, Clarke2021,higashitani2022quadratic, Mohammadi2019}. 
We investigate those giving rise to toric degenerations of $\Fl({K};n)$ by constructing combinatorial mutations between them. We note that the matching field $B^{w_0}$ for $w_0 = (n,n-1, \dots, 1)$ and ${K} \subseteq [n]$ is the diagonal matching field which classically gives rise to a toric degeneration of $\Fl({K};n)$ \cite{Miller2004}. The corresponding matching field polytopes 
$\mathcal{P}_{w_0}^{{K}}$ are the well-studied Gelfand-Tsetlin polytopes. 
{We describe a family of matching fields indexed by permutations avoiding specific patterns, given in \Cref{def : permutation pattern}, that gives rise to toric degenerations of $\Fl({K};n)$ as follows.}

\begin{restatable}{thm}{DSRTwoFlagVarieties}
\label{thm: DSR_2 Flag varieties}
{Fix $n \in \ZZ$, $n>0$ and a subset ${K} \subseteq [n]$.} If $\sigma\in S_n$ is a permutation that avoids the patterns $4123, 3124, 1423$ and $1324$, then the polytope {$\mathcal{P}^{{K}}_{\Bsig}$ }is combinatorial mutation equivalent to the Gelfand-Tsetlin polytope $\mathcal{P}_{w_0}^{{K}}$.
% $\mathcal{P}^{[n]}_{w_0}$. 
It follows that $B^\sigma$ gives rise to a toric degeneration of {$\Fl({K};n)$}. {In particular, by taking ${K}=\{k\}$ or ${K}=[n]$, the matching field $B^\sigma$ gives rise to a toric degeneration of $\Gr(k,n)$ or $\Fl_n$, respectively.}
\end{restatable}

Our main tool for understanding matching field polytopes for partial flag varieties is the Minkowski sum property. In general, the matching field polytopes for Grassmannians $\mathrm{Gr}(k,n)$ for $k\in {K}\subseteq [n]$ can be used to build the matching field polytopes for the partial flag variety $\Fl({K};n)$.

\begin{restatable}[Minkowski~sum~property]{prop}{Minkowskisums}\label{Minkowski}
Let ${K} \subseteq [n]$ be a non-empty subset and $X_\Lambda^{{K}}\subseteq \prod_{k \in {K}} \PP^{\binom{n}{k} - 1}$ be the toric subvariety associated to a matching field $\Lambda^{{K}}$.~Then the polytope $\mathcal{P}^{{K}}_\Lambda$~of~$X^{{K}}_\Lambda$~is~equal to the Minkowski sum of the polytopes $\mathcal{P}^{k}_{\Lambda}$ associated to the matching fields $\Lambda^{k}$ for~each~$k \in {K}$.
\end{restatable}
As mentioned above, given a closed subvariety of an algebraic torus $X =V(\MI)$, the initial ideals associated to the top-dimensional cones of $\Trop(X)$ are good candidates to give toric degenerations \cite[Lemma~1]{bossinger2017computing}. A maximal cone in $\Trop(X)$ is called prime if it gives rise to a toric degeneration of $X$. In \cite{escobarwall}, the authors study the relationship between polytopes arising via toric degenerations from the interior of two adjacent prime cones in $\Trop(\Gr(2,n)^{\circ})$. It is shown that these polytopes are related by piecewise linear maps and, in the appendix by Ilten, it is shown that the linear maps are combinatorial mutations, which are analogous to those we study. As noted in \cite[Example 4.65]{bossinger2021families}, these mutations can be derived from the \textit{Fock–Goncharov tropicalization of cluster mutation} \cite[Definition~1.22]{gross2014canonical}.
More specifically, the tropical maps that arise from cluster mutation have the property that their factor polytope is a line segment. We note that all mutations that we study have this property as well. See \Cref{def flag tropical map}.
\medskip

\noindent
\textbf{Computational Results.}
In \Cref{sec: computational results}, we compute tropicalizations of  $\Gr(3,6)$, $\Gr(3,7)$, $\Fl_4$~and~$\Fl_5$, and highlight the new families of toric degenerations obtained by our methods. 
We introduce the matching fields $B_c^\sigma$ where $\sigma \in S_n$ is a permutation and $c$ is a positive integer. If $c = 1$ then $B_c^\sigma = B^\sigma$ is the matching field studied earlier in the paper. If $\sigma = w_0$ then $B_c^\sigma$ leads to the Gelfand-Tsetlin degeneration for~each~$c$. 

We follow the same labelling scheme for maximal cones of tropical Grassmannians as in \cite{speyer2004tropical,herrmann2009draw}.
The computational results show that all maximal cones of $\Trop(\Gr(3,6)^{\circ})$ and many of the of the maximal cones of $\Trop(\Gr(3,7)^{\circ})$ that give rise to toric degenerations can be described by a matching field $B_c^\sigma$. Note that not all maximal cones can be described in this way, for instance those containing a sub-weight of type EEEE, which cannot be obtained from any matching field \cite{Mohammadi2019}. 
We summarise the computational results for $\Gr(3,7)$ in \Cref{example: gr37 all cones} and \Cref{tab: gr37 matching field polytopes}. 
As a result, of the $125$ isomorphism classes of maximal cones of $\Trop(\Gr(3,7)^{\circ})$ we have that $69$ of them are distinct prime cones and, of those, $40$ can be obtained from matching fields $B_c^\sigma$. Of these $40$, we have that $20$ of them can be obtained when $c = 1$.
In addition we have studied the adjacency of these prime cones. Our calculations show that a combinatorial mutation, described in the proof of \Cref{thm: DSR_2 Flag varieties}, corresponds to either staying in the same cone or moving to an adjacent cone, that is, to a cone sharing a codimension~one face.

For the full flag variety, the toric degenerations associated to each of the four prime cones of $\Trop(\Fl_4^{\circ})$ can be obtained by a matching field of the form $B_c^\sigma$. For $\Fl_5$, we obtain $22$ distinct toric degenerations. See \Cref{table: flag 4 matching fields} and \Cref{table: flag 5 matching fields}. 

\medskip

\noindent{\bf Structure of the paper.}
In \Cref{sec: toric subvarieties}, we apply tools from toric geometry to find properties of matching field polytopes. 
In particular, we detail a method to obtain the polytopes of toric subvarieties of a general smooth projective toric variety associated with a lattice polytope. See \Cref{proc: toric subvariety}. We then apply this procedure to the toric varieties associated to matching field polytopes for $\Fl({K};n)$.
In \Cref{Mink}, we show how these polytopes can be viewed as Minkowski sums of matching field polytopes of Grassmannians. In \Cref{sec: toric degen from matching fields}, we prove \Cref{thm:toric degen partial flag} by showing that all matching field polytopes are normal. See \Cref{prop: lattice points of flag polytopes}.
We then study the Pl\"ucker algebras from the point-of-view of SAGBI basis theory. See \Cref{cor:SAGBI}.

In \Cref{sec: matching fields from permutations}, we introduce the family of matching fields $\Bsig$ indexed by permutations $\sigma \in S_n$. This is a family of coherent matching fields that are simultaneously defined for all partial flag varieties $\Fl({K};n)$. In \Cref{flag variety section}, we prove \Cref{thm: DSR_2 Flag varieties} by showing that all matching field polytopes $\mathcal{P}_{\Bsig}^{{K}}$ are combinatorial mutation equivalent by introducing a special family of tropical maps. See \Cref{def flag tropical map}. 
In \Cref{sec: computational results}, we provide computational results that extend the family of matching fields $\Bsig$. See \Cref{def: scale the second row weight matrix}. We give complete computations for Grassmannians $\Gr(3, 6)$ and $\Gr(3,7)$ in \Cref{table: gr36 polytopes,tab: gr37 matching field polytopes}, along with the full flag varieties $\Fl_4$ and $\Fl_5$ in \Cref{table: flag 4 matching fields,table: flag 5 matching fields}.

\medskip
\noindent{\bf Acknowledgement.}
The authors were supported by the grants G0F5921N (Odysseus programme) and G023721N from the Research Foundation - Flanders (FWO), and the KU Leuven grant iBOF/23/064. F.M. was partially supported by the UiT Aurora project MASCOT. F.Z. was also partially supported by the FWO fundamental research fellowship (1189923N).
{We thank the reviewer for helpful comments and suggestions.
}
 
\section{Toric subvarieties and the Minkowski sum property}\label{sec: toric subvarieties}

In this section, we set up notation and explain our method for studying toric varieties and their polytopes. In particular, \Cref{proc: toric subvariety} gives an explicit way for constructing such polytopes. In \Cref{sec: matching fields}, we introduce matching fields and in \Cref{recipe} we apply \Cref{proc: toric subvariety} to toric subvarieties of multiprojective spaces arising from matching fields. In \Cref{Mink}, we give a construction for the matching field polytopes of $\Fl_n$ as Minkowski sums of matching field polytopes for the Grassmannians. 

\subsection{Closed toric subvarieties of a projective toric variety}\label{discussion}
Here, we derive a general procedure for constructing toric subvarieties of a given projective toric variety. For further details, we refer to \cite[Chapter 5]{Cox2011} and \cite[Chapter 10]{Miller2004}.

Fix a lattice $M \cong \ZZ^d$ and let $\mathcal{P} \subseteq M_\RR$ be a full-dimensional lattice polytope, i.e., the vertices $V(\mathcal{P}) \subseteq M$ are lattice points. {We denote by $\NN$ the set of non-negative integers $\ZZ_{\geq 0}$.} The polytope $\mathcal{P}$ is said to be \textit{normal} if for any $k \in \NN$ and any lattice point $\alpha \in (k\mathcal{P} \cap M)$ there exist lattice points $\alpha_1, \dots, \alpha_k \in (\mathcal{P} \cap M)$ such that $\alpha = \alpha_1 + \cdots + \alpha_k$. Throughout, we will assume that $\mathcal{P}$ is a normal lattice polytope and study its associated toric variety. Let $\{p_1, \dots, p_s\} = (\mathcal{P} \cap M)$ be the lattice points of $\mathcal{P}$. The toric variety $Y_\mathcal{P} \subseteq \PP^{s - 1}$ associated to $\mathcal{P}$ is the Zariski closure of the image of the map
\[
(\CC^*)^d \rightarrow \PP^{s-1} \colon t \mapsto \left( \chi^{p_1}(t), \dots, \chi^{p_s}(t) \right)
\]
where $\chi^{(m_1,\dots, m_d)}(t_1, \dots, t_d) := t_1^{m_1} t_2^{m_2} \cdots t_d^{m_d}$ is a \textit{character of the torus $(\CC^*)^d$}. The characters of $(\CC^*)^d$ naturally form a lattice with the group operation $(\chi^{m_1} \cdot \chi^{m_2})(t) := \chi^{m_1}(t) \chi^{m_2}(t) = \chi^{m_1 + m_2}(t)$. So the character lattice is naturally isomorphic to $M$ and by abuse of notation we will identify these lattices.

The homogeneous coordinate ring of $Y_\mathcal{P}$ is given by $\CC[C(\mathcal{P}) \cap (M \times \ZZ)]$ where 
\[
C(\mathcal{P}) := {\rm Cone}(\mathcal{P} \times \{1\}) = \{(\lambda p, \lambda) \in M_\RR \times \RR : \lambda \ge 0, \ p \in \mathcal{P} \} \subseteq M_\RR \times \RR.
\]

\subsubsection{The total coordinate ring and its grading}\label{sec:total}

Let $Y_\mathcal{P}$ be a smooth projective toric variety associated to a normal lattice polytope $\mathcal{P}\subseteq M_\RR$, where $M\cong \ZZ^{\rk(M)}$ is the character lattice of the torus of $Y_\mathcal{P}$. The goal of this subsection is to give a uniform description of toric subvarieties of $Y_\mathcal{P}$. Recall that $Y_\mathcal{P} = \Proj (\CC[C(\mathcal{P})\cap (M\times \ZZ)])$. Therefore, any toric subvariety of $Y_\mathcal{P}$ will be given by a monomial subalgebra of $\CC[C(\mathcal{P})\cap (M\times \ZZ)]$. Our approach to find these subalgebras is as follows. We associate to $Y_\mathcal{P}$ a particularly nice polynomial ring $R$ and obtain toric subvarieties of $Y_\mathcal{P}$ from toric ideals in $R$ satisfying a certain homogeneity condition.

We start with a set of inequalities defining the polytope $\mathcal{P}$. Let $f$ be the number of facets of $\mathcal{P}$. An \textit{H-description} of $\mathcal{P}$ is a collection of linear inequalities that define $\mathcal{P}$ as follows:
\[   x \in \mathcal{P} \iff \mathbf{B} \cdot x + \rho \geq \underline 0 \]
where $\mathbf B \in \ZZ^{f \times \rk(M)}$ is a full-rank matrix whose rows are primitive inner-normal vectors to the facets of $\mathcal{P}$, and $\rho \in {\RR^f}$ and $x \in M_\RR$ are column vectors.
We may think of $\mathbf B$ as a map from $M$ to $\ZZ^f$. Since $\mathcal{P}$ is full-dimensional, the columns of $\mathbf B$ are linearly independent. Hence, $\mathbf B$ embeds the character lattice $M$ into $\ZZ^f$. 
Since $Y_\mathcal{P}$ is smooth, we have that the only lattice points in $\mathbf B \cdot M_\RR\cap \ZZ^f$ are of the form $\mathbf B \cdot x$ for some $x \in M$. See \cite[Chapter 4]{Cox2011}. So the embedding can be completed to a short exact sequence of lattices
\begin{equation}
\begin{tikzcd}\label{picard seq}
0 \arrow[r] & M \arrow[r,"\mathbf B"] & \ZZ^f \arrow[r,"\mathbf D"] & {\Cl(Y_\MP)} \arrow[r] & 0.
\end{tikzcd}
\end{equation}
The geometric interpretation of this short exact sequence is given by: $\mathbf B$ is the map sending a character to its corresponding torus-invariant Weil divisor; and $\mathbf D$ is the map that takes a torus-invariant Weil divisor to its class in the class group {$\Cl(Y_\MP)$} of $Y_\mathcal{P}$.

\begin{definition}\label{def: total coordinate ring}
The \emph{total coordinate ring} $R$ of a smooth projective toric variety $Y_\mathcal{P}$ is the polynomial ring $\CC[z_1,\dots,z_f]$ with variables indexed by the facets of $\mathcal{P}$.
\end{definition}

The ring $R$ is graded by {$\Cl(Y_\MP)$}. For each $m \in \ZZ^f$, the monomial $x^m \in R$ has degree $\deg(x^m)= \mathbf D \cdot m \in {\Cl(Y_\MP)}$. The variety $Y_\mathcal{P}$ can be described purely in terms of the total coordinate ring.
We may rewrite the semigroup ring as follows: 
\[ \CC[C(\mathcal{P})\cap (M\times \ZZ)]=\CC\left[\bigoplus_{k=0}^{\infty}(k\mathcal{P}\cap M)\times \{k\}\right]. \]

The lattice points $(k\mathcal{P} \cap M)$ of each dilation $k\mathcal{P}$ can be viewed as a graded piece of $R$. Fix $k \in \NN$ and let 
\[ \iota : M_\RR \hookrightarrow \RR^f, \quad x \mapsto \mathbf{B} \cdot x + k \rho \]
be a translation of the embedding $\mathbf B$. By the definition of the H-description of $\mathcal{P}$, we have that a point $x \in M_\RR$ lies in $k\mathcal{P}$ if and only if $\iota(x) \in \RR_{\ge 0}^f$, that is $\mathbf B\cdot x + k\rho \ge \underline 0$. By the short exact sequence in \Cref{picard seq}, we have that
$\iota(k\mathcal{P})= k(\mathbf B \cdot \mathcal{P} + \rho) = \mathbf D^{-1}(k \underline d)\cap \RR^f_{\geq 0}$, where $\underline d = \mathbf D \cdot \rho$. In particular, $\iota$ gives a bijection between lattice points in $k\mathcal{P}$ and $k\underline d$-graded monomials in $R$, so that we have $Y_\mathcal{P} = \Proj\left(\CC\left[\bigoplus_{k=0}^{\infty}(k\mathcal{P}\cap M)\times \{k\}\right]\right)\cong \Proj\left( \bigoplus_{k=0}^{\infty}(R_{k\underline d})\right)$.

\subsubsection{Toric subvarieties from monomial maps}
The description of the coordinate ring of $Y_\mathcal{P}$ in terms of the total coordinate ring, in the previous section, allows us to generalise the correspondence between homogeneous ideals and closed subvarieties of projective spaces. Using the notation of Subsection~\ref{sec:total} we have the following.

\begin{lemma}[\cite{Cox2011}, Proposition~5.2.4]
Let $R$ be the total coordinate ring of a projective toric variety $Y_\mathcal{P}$ of a full dimensional lattice polytope $\mathcal{P}$. Then any {$\Cl(Y_\MP)$}-homogeneous ideal $\MI\subset R$ gives rise to a closed subvariety of $Y_\mathcal{P}$. Moreover, all closed subvarieties of $Y_\mathcal{P}$ arise in this way.  
\end{lemma}
\noindent
In particular, any toric subvariety of $Y_\mathcal{P}$, that is, any closed subvariety of $Y_\mathcal{P}$ which is a toric variety with respect to a subtorus of the torus $T$, arises from a toric ideal in $R$, i.e., the kernel %$\ker(\phi)$ 
of some~monomial~map 
\[\phi:\CC[z_1,\dots,z_f]\rightarrow \CC[p_1,\dots,p_k] \]
which is concurrent with the $\Cl(Y_\MP)$-grading of $R$. By this we mean that if $\mathbf{z}^u - \mathbf{z}^v \in \ker(\phi)$ then $\mathbf{z}^u$ and $\mathbf{z}^v$ have the same $\Cl(Y_\MP)$-degree, i.e., $\mathbf D \cdot u = \mathbf D \cdot v$. Consider the lattice $\ZZ^k$ associated to the ring $\CC[p_1,\dots,p_k]$ and define \begin{equation}\label{eq:phi_hat_map}
\hat\phi: \ZZ^f \to \ZZ^k
\end{equation}
to be the map of lattices corresponding to $\phi$, i.e., $\phi(\mathbf{z}^u) = \mathbf{p}^v$ if and only if $\hat\phi(u) = v$. By an abuse of notation, we will use $\hat\phi$ to denote the map of spaces $\RR^f \rightarrow \RR^k$ that naturally extends $\hat\phi$.
Given such a monomial map $\phi$, we construct a closed toric subvariety $X_\phi\subseteq Y_\mathcal{P}$ which is the vanishing locus in $Y_\mathcal{P}$ of the kernel $\ker(\phi)$. We will now show that $X_\phi$ is {the} toric variety arising from the polytope $\mathcal{P}_\phi = \hat{\phi}(\mathcal{P})$ viewed as a polytope with respect to the lattice $\hat{\phi}(M)$.

\begin{restatable}{prop}{toricsubvarieties}
\label{prop: toric subvarieties description}
Let $Y_\mathcal{P}$ be a smooth projective toric variety with lattice polytope $\mathcal{P}$ and total coordinate ring $R=\CC[z_1,\dots,z_f]$. Fix a monomial map $\phi:\CC[z_1,\dots,z_f]\to \CC[p_1,\dots,p_k]$ which is concurrent with the {$\Cl(Y_\MP)$}-grading of $R$. Let $\hat{\phi}$ be the map of lattices defined above.
Then $X_\phi$ is given by the toric variety $X_{\mathcal{P}_\phi}$ constructed from the polytope $\mathcal{P}_\phi=\hat{\phi}(\mathcal{P})$.
\end{restatable}

\begin{proof}
We begin by showing that $X_\phi$ is given by $\Proj \left(\bigoplus_{k=0}^{\infty} (\Ima \phi)_{k\underline d}\right)$.
Consider the natural surjection $R\twoheadrightarrow R/\ker(\phi) \cong \Ima(\phi)$. Since $\phi$ is a monomial map, we have an isomorphism $\Ima(\phi) \cong \CC[\NN S]$ where $S=\{s_1,\dots,s_f\}\subset \RR^k$ is the image of the set of unit vectors in $\RR^f$. By assumption, $\phi$ is concurrent with the {$\Cl(Y_\MP)$}-grading of $R$ and so $\ker(\phi)$ is {$\Cl(Y_\MP)$}-homogeneous. We have that $R\twoheadrightarrow R/\ker(\phi)$ restricts to a surjection on each {$\Cl(Y_\MP)$}-graded piece. In particular, we have that
\[
    \bigoplus_{k=0}^{\infty}(R_{k\underline d})\twoheadrightarrow 
    \bigoplus_{k=0}^{\infty}(R_{k\underline d})/(\ker \phi)_{k\underline d}\cong
    \bigoplus_{k=0}^{\infty} (\Ima \phi)_{k\underline d}. 
\]
By \cite[Proposition~13.8]{Goertz2020}, the surjection above induces a closed immersion 
\[
\Proj \left(\bigoplus_{k=0}^{\infty} (\Ima \phi)_{k\underline d} \right)
\rightarrow 
\Proj \left(\bigoplus_{k=0}^{\infty}(R_{k\underline d})\right)
=Y_\mathcal{P}.
\] 
Next, we show that $\Proj \left(\bigoplus_{k=0}^{\infty} (\Ima \phi)_{k\underline d}\right)$ coincides with toric variety $X_{\mathcal{P}_\phi}$ associated to the polytope $\mathcal{P}_\phi := \hat\phi(\mathcal{P})$. By the construction of  $X_{\mathcal{P}_\phi}$, we have 
\[
X_{\mathcal{P}_\phi}=\Proj\left( \CC[C(\mathcal{P}_\phi)\cap (\ZZ S \times \ZZ)]\right)=
\Proj \left(\CC\left[\bigoplus_{k=0}^{\infty} (k\mathcal{P}_\phi) \cap \ZZ S\right]\right).
\]  
Since $\phi$ is concurrent with the {$\Cl(Y_\MP)$}-grading of $R$, any lattice point of $k\mathcal{P}_\phi$ corresponds to the image under $\phi$ of a $k\underline d$-graded monomial. It follows that 
\[
\CC\left[\bigoplus_{k=0}^{\infty} (k\mathcal{P}_\phi) \cap \ZZ S\right]\cong \bigoplus_{k=0}^{\infty} (\Ima \phi)_{k\underline d}.
\]
We conclude the proof by taking $\Proj$ of the above.
\end{proof}

In the following procedure, we summarise  the main discussion of this subsection. That is, for any projective toric variety $Y_\mathcal{P}$ associated to a polytope $\mathcal{P}$, we describe the polytope associated to a toric subvariety $X_\phi$ of $Y_\mathcal{P}$, where $X_\phi$ is the vanishing locus of a monomial map $\phi$.

\begin{procedure}\label{proc: toric subvariety}
Let $Y_\mathcal{P}$ be a smooth toric variety associated to a full-dimensional lattice polytope $\mathcal{P}\subset M_\RR$ with $f$ 
facets. Consider the lattice $\ZZ^f$ and let $R=\CC[\NN^f]$ be the total coordinate ring of $Y_\mathcal{P}$. Fix a monomial map $ \phi: R \to \CC[p_1,\dots,p_k]$.
The toric subvariety $X_\phi$ of $Y_\mathcal{P}$ can be constructed through the following procedure:
\begin{enumerate}
    \item\label{item: h-description} Compute the H-description of $\mathcal{P}$. That is, calculate the matrix $\mathbf B: M \to \ZZ^f$ and vector $\rho \in \RR^f$ so that: the rows of $\mathbf B$ are primitive vectors; and $\mathbf B\cdot x + \rho \geq \underline 0$ if and only if $x \in \mathcal{P}$. Identify $\mathcal{P}$ with $\{\mathbf B \cdot x + \rho \mid x \in \mathcal{P} \} \subseteq \RR^f$.
    
    \item\label{item: grading matrix} Calculate the grading matrix $\mathbf D$ of $R$ by completing the short exact sequence 
    \[
    \begin{tikzcd}
    0 \arrow[r] & M \arrow[r,"\mathbf B"] &\ZZ^f \arrow[r,"\mathbf D"] & {\Cl(Y_\MP)}\arrow[r] &0.
    \end{tikzcd}
    \]
    Define $\underline d = \mathbf D\cdot \rho \in {\Cl(Y_\MP)}$ and check that $\ker(\phi)$ is homogeneous with respect to the {$\Cl(Y_\MP)$}-grading.
    
    \item\label{item: x_phi} Consider $X_\phi =\Proj\left(\bigoplus_{k=0}^{\infty} (\Ima \phi)_{k\underline d}\right)$. Let $\hat\phi : \ZZ^f \rightarrow \ZZ^k$ be the lattice map corresponding to $\phi$. The polytope of $X_\phi$ is given by $\mathcal{P}_{\phi} := \hat{\phi}(\mathcal{P})$. The polytope $\mathcal{P}_\phi$ is a lattice polytope with respect to the lattice $\ZZ S$ where $S$ is the image of the set of unit vectors in $\RR^f$ under $\hat\phi$.
\end{enumerate}
\end{procedure}

\subsubsection{Embeddings into projective space}
To conclude our preliminary discussion, we note that some readers may be more familiar with a description of $X_\phi\subseteq Y_\mathcal{P} \subset \PP^{s-1}$ as a closed toric subvariety of some high-dimensional projective space. We now briefly confirm the description of the lattice polytope associated to $X_\phi$ described in \Cref{proc: toric subvariety} in this setting.

Let $f$ be the number of facets of $\mathcal{P}$. Recall that the toric variety $Y_\mathcal{P}$ embeds as a subvariety of $\PP^{s-1}$ where $s$ is the number of lattice points of $\mathcal{P}$. By this description, we have $Y_\mathcal{P}=\Proj\left( \CC[t_1,\dots, t_s]/\ker(\psi) \right)$, where $\psi$ is a monomial map
\[
    \psi: \CC[t_1,\dots, t_s] \rightarrow \CC[z_1, \dots, z_f], \quad t_i \mapsto \mathbf{z}^{v_i}
\]
and $v_i\in M$ are the lattice points of $\mathcal{P}$. To embed $X_\phi$ into $\PP^{s-1}$, we compose $\psi$ with $\phi$. The ideal $\ker(\phi \circ \psi)$ defines $X_\phi$ as a subvariety of $\PP^{s-1}$. Explicitly, by the ring isomorphism theorems we have
\[
\CC[t_1,\dots, t_s]/\ker(\phi \circ \psi)\cong \Ima(\psi) / \Ima( \psi|_{\ker(\phi \circ \psi)}).
\] 
Finally, $\Ima(\psi) /\Ima( \psi)|_{\ker(\phi \circ \psi)}$ is isomorphic to $
\bigoplus_{k=0}^{\infty}R_{k\underline d}/(\ker(\phi)\cap \bigoplus_{k=0}^{\infty}R_{k\underline d})\cong \bigoplus_{k=0}^{\infty} \CC[\NN S]_{k\underline d}$. So, by applying $\Proj$ to the above, we see that this description agrees with the one given in \Cref{proc: toric subvariety}.

\subsection{Matching fields} \label{sec: matching fields}

\begin{definition}\label{def coherent matching field}
A \textit{matching field for $\mathrm{Gr}(k,n)$} is a map $\Lambda: \binom{[n]}{k}\to S_k.$
A \textit{matching field for {{$\Fl({K};n)$}}} is a map
\[
    {\Lambda: \{ I \mid \varnothing \subsetneq I \subsetneq [n], |I|\in {K}\} \to \bigsqcup_{k \in {K}} S_k}
\]
such that for each $k \in [n-1]$ the restriction of $\Lambda$ to $\binom{[n]}{k}$ is a matching field for $\Gr(k, n)$. A matching field $\Lambda$ is \textit{coherent} if there exists a matrix $M \in \RR^{(n-1)\times n}$ such that for each subset $I$ in the domain of $\Lambda$ with $I=\{i_1< \dots <i_k\}$ we have 
\[
    \Lambda(I)=\argmin_{\sigma \in  S_k}\sum\limits_{a=1}^{k} M_{a,i_{\sigma(a)}}
\]
and the minimum is attained at a unique $\sigma \in S_k$. Each coherent matching field $\Lambda$ for {$\Fl({K},n)$} gives rise to a monomial map
\begin{equation}\label{eq: matching field monomial map}
    \phi_\Lambda: R_\MP \rightarrow \CC[x_{i,j} \mid \ i\in [n-1], j \in [n]], \quad 
p_{\{j_1,\dots,j_k\}} \mapsto \sgn(\Lambda(I))\prod_{i=1}^k x_{i,j_{\Lambda(I)(i)}}.
\end{equation}
\end{definition}
In the next subsection, we will study the polytopes and toric varieties associated to matching fields.

\subsection{Multiprojective toric varieties from matching field ideals}\label{recipe}

In this section, we will apply \Cref{proc: toric subvariety} to the setting of toric degenerations of partial flag varieties $\Fl({K};n)$ from matching fields. We begin by fixing our notation and defining the toric varieties that will appear throughout the rest of the paper. For simplicity, we will work with the full flag variety $\Fl_n$ however the same method may be applied to any partial flag variety $\Fl({K};n)$.

\medskip

\noindent \textbf{Notation.}
For each $k\in [n]$, {we use $\binomial{n}{k}$ to denote the number of $k$-subsets of $[n]$}. For each $d \in \NN$, we let $\Delta_d \subseteq \RR^{d}$ be the $d$-dimensional simplex given by the convex hull of the standard basis vectors of $\RR^{d}$ together with the origin. We define the polytope $\mathcal{P} = \Delta_{\binomial{n}{1}-1} \times \cdots \times \Delta_{\binomial{n}{n-1} - 1}$ to be the product of simplicies and define $Y_{\mathcal{P}}$ to be the toric variety associated to $\mathcal{P}$. In the case of the partial flag variety $\Fl({K};n)$ where {$K = \{k_1 < \dots < k_\ell \}$}, we consider instead the polytope $\mathcal{P}^{{K}} = \Delta_{\binomial{n}{{k_1}} - 1} \times \dots \times \Delta_{\binomial{n}{{k_\ell}}-1}$.

\medskip

\noindent{\bf Product of polytopes.}
Given two polytopes $\mathcal{P}$ and $\mathcal{Q}$ with $s$ and $t$ facets respectively, their product $\mathcal{P} \times \mathcal{Q}$ has $s + t$ facets. Explicitly, if $F_1, \dots, F_s$ are the facets of $\mathcal{P}$ and $G_1, \dots, G_t$ are the facets of $\mathcal{Q}$, then the facets of $\mathcal{P} \times \mathcal{Q}$ are given by $F_i \times \mathcal{Q}$ along with $\mathcal{P} \times G_j$ where $i \in [s]$ and $j \in [t]$. So, each facet of $\mathcal{P}$ and $\mathcal{Q}$ gives rise to a facet of $\mathcal{P} \times \mathcal{Q}$.

\medskip

Given a polytope $\mathcal{P}$, recall from \Cref{def: total coordinate ring} that the total coordinate ring of $Y_\mathcal{P}$, which we write as $R_\mathcal{P}$, is a polynomial ring that has one variable for each facet of $\mathcal{P}$.
For each $k \in [n-1]$ the simplex $\Delta_{\binomial{n}{k} - 1}$ has $\binomial{n}{k}$ facets, so we may write
\[
R_{\Delta_{\binomial{n}{k}- 1}} =
\CC\left[p_I \mid I\subseteq [n], \ |I| = k 
\right] 
\text{ and }
R_{\mathcal{P}} = 
R_{\Delta_{\binomial{n}{1}-1}} \otimes_\CC \cdots \otimes_\CC R_{\Delta_{\binomial{n}{n-1}-1}} = 
\CC\left[p_I  \mid \varnothing \subsetneq I \subsetneq [n] \right]. 
\]
Each polynomial ring $R_{\Delta_{\binomial{n}{k}- 1}}$ is the homogeneous coordinate ring of $\PP^{\binomial{n}{k} - 1}$. Recall that the Pl\"{u}cker embedding $\Fl_n \hookrightarrow \PP^{\binomial{n}{1}-1}\times\dots\times \PP^{\binomial{n}{n-1}-1}$ is defined by the ring map
\begin{equation}\label{eq: flag plucker ring map}
\phi: R_\MP \to \CC[x_{i,j}|\ i\in [n-1], j \in [n]], \quad 
    p_I \mapsto \det(X_I)
\end{equation}
where $X= (x_{i,j})$ is an $(n-1) \times n$ matrix of variables and for each subset $I \subset [n]$ we write $X_I$ for the submatrix of $X$ whose columns are those indexed by $I$ and whose rows are indexed by $1, \dots, |I|$. A matching field can be thought of as a modification of $\phi$ to a monomial map which sends each Pl\"ucker variable $p_I$ to a single term of $\det(X_I)$. The kernel of this monomial map, as in \eqref{def coherent matching field} gives rise to a toric variety which, under certain conditions, is a toric degeneration of the flag variety.

\medskip

We proceed to apply \Cref{proc: toric subvariety} to the toric variety $Y_\MP$ with total coordinate ring $R_\MP$ together with the monomial map $\phi_\Lambda$ in \eqref{def coherent matching field} for some coherent matching field $\Lambda$.

\medskip

\noindent \textbf{Step~\ref{item: h-description}.} Let us compute the H-description of $\MP$. For each $k \in \NN$ we write $\mathbf I_k$ for the $k \times k$ identity matrix and $\mathbf{1}_k$ for the column vector of all ones. We define the $(k+1) \times k$ matrix $\mathbf{J}_k = [\mathbf{I}_k  \mid -\textbf{1}_k]^T$ given by the identity matrix with a row of $-1$'s below it. We define the column vector $\rho_{k+1}=(0, \dots,0,1)^T \in \RR^{k+1}$. The simplex $\Delta_k$ is defined to be the convex hull of the standard basis vectors in $\RR^k$ together with the origin. Equivalently, we have that 
\[
\Delta_k = \{ (x_1, \dots, x_k) \in \RR^k \colon x_i \ge 0 \text{ for all $i$ and } x_1 + \cdots + x_k \le 1 \}.
\]
Therefore the H-description of $\Delta_k$ is given by $\Delta_k = \{x \in \RR^k \colon \mathbf J_k \cdot x \ge -\rho_{k+1} \}$. Since $\mathcal{P}$ is the direct product of the simplices $\Delta_{\binomial{n}{1} - 1}, \dots, \Delta_{\binomial{n}{n-1} - 1}$, we define the matrix $\mathbf B$ to be the direct sum of $\mathbf J_{\binomial{n}{1} - 1}, \dots, \mathbf J_{\binomial{n}{n-1} - 1}$ which is given by
\[
\mathbf B = 
\begin{bmatrix}
\mathbf{J}_{\binomial{n}{1}-1}  & 0                   & \cdots & 0      \cr
0                   & \mathbf{J}_{\binomial{n}{2}-1}  & \cdots & 0      \cr
\vdots              & \vdots              & \ddots & \vdots \cr
0                   & 0                   & \cdots & \mathbf{J}_{\binomial{n}{n-1}-1} \cr
\end{bmatrix} 
\in \ZZ^{(\binomial{n}{1}+\dots+\binomial{n}{n-1}) \times (\binomial{n}{1}-1+\dots+\binomial{n}{n-1}-1)}.
\] 
Similarly, we define the column vector
$\rho = [\rho_{\binomial{n}{1}}^T|\dots|\rho_{\binomial{n}{n-1}}^T]^T \in \ZZ^{\binomial{n}{1} + \cdots + \binomial{n}{n-1}}$. Hence, by construction we have that the H-description of $\mathcal{P}$ is $\{x \in \RR^{\binomial{n}{1} - 1 + \cdots + \binomial{n}{n-1}-1} \colon \mathbf B \cdot x \ge - \rho \}$.

\medskip \noindent \textbf{Step~\ref{item: grading matrix}.} Let us compute the grading matrix $\mathbf D$, which is given by the cokernel of $\mathbf{B}$. Observe that for each $k \in \NN$, the cokernel of $\mathbf J_k$ is $\ZZ^{k+1} / \Ima(\mathbf J_k)$ which we identify with $\ZZ$ by sending $(x_1, \dots, x_{k+1}) \in \ZZ^{k+1}$ to $x_1 + \dots + x_{k+1} \in \ZZ$. We obtain a short exact sequence given by $\ZZ^k \rightarrow \ZZ^{k+1} \rightarrow \ZZ$ where the left-hand map is given by $\mathbf J_k$ and the right-hand map is the all-ones matrix $\mathbf 1_{k+1}^T$. So, we define
\[
\mathbf D = 
\begin{bmatrix}
\mathbf{1}_{\binomial{n}{1}}^T     & 0 & \cdots & 0 \cr
0 & \mathbf{1}_{\binomial{n}{2}}^T  & \cdots & 0 \cr
\vdots & \vdots & \ddots & \vdots \cr
0  & 0 & \cdots & \mathbf{1}_{\binomial{n}{n-1}}^T \cr
\end{bmatrix}\in \ZZ^{(n-1) \times (\binomial{n}{1} + \dots + \binomial{n}{n-1})}.
\]
Hence, we obtain a short exact sequence $\ZZ^{\binomial{n}{1} - 1 + \cdots + \binomial{n}{n-1} - 1} \to \ZZ^{\binomial{n}{1} + \cdots + \binomial{n}{n-1}} \to \ZZ^{n-1}$ with maps given by the matrices $\mathbf B$ and $\mathbf D$. We also note that $\underline d := \mathbf D \cdot \rho = \mathbf{1}_{n-1}\in \ZZ^{n-1}=: {\Cl(Y_\MP)}$.

Explicitly, the grading induced by $\mathbf D$ on the total coordinate ring $R_\mathcal{P} = \CC[p_I : \varnothing \subsetneq I \subsetneq [n]]$ is given by  $\deg(p_I) = e_{|I|} \in \ZZ^{n-1}$ for each subset $I$, where $e_1, \dots, e_{n-1}$ are the standard unit vectors in $\ZZ^{n-1}$. {We also define the $\Cl(Y_\MP)$-grading on the ring $\CC[x_{i,j}]$ by {$\deg(x_{i,j}) = e_i - e_{i-1}$ if $i > 1$ and $\deg(x_{1,j}) = e_1$, for all $j \in [n]$}.} We will now show that the matching field ideal $\ker(\phi_\Lambda)$ 
is homogeneous with respect to the grading given~by~$\mathbf D$.

\begin{lemma}\label{lemma: grading}
The monomial map $\phi_\Lambda$ is {concurrent} with respect to the {$\Cl(Y_\MP)$}-grading of $R_\MP$. 
\end{lemma}

\begin{proof}
The result follows immediately from the fact that the grading on the total coordinate ring $R_\mathcal{P}$ is induced by the {$\Cl(Y_\MP)$}-grading on $\CC[x_{i,j}]$. Explicitly, if $I \subseteq [n]$ is a $k$-subset such that $\phi_\Lambda(p_I) = cx_{1,i_1}x_{2,i_2} \dots x_{k,i_k}$ for some $i_j \in [n]$ and $c \in \{+1, -1 \}$, then we have that
\[
\deg(\phi_\Lambda(p_I)) =
\deg(x_{k,i_k}) + \dots + \deg(x_{2,i_2}) + \deg(x_{1,i_1}) =
(e_k - e_{k-1}) + \dots + (e_2-e_1) + (e_1) =
e_k = \deg(p_I). \qedhere
\]
\end{proof}

\medskip \noindent \textbf{Step~\ref{item: x_phi}.} For the final step of \Cref{proc: toric subvariety}, we will identify $\mathcal{P}$ with its image under the map $x \mapsto \mathbf B \cdot x + \rho$ which lives in $\RR^{\binomial{n}{1} + \dots + \binomial{n}{n-1}}$. Consider the lattice $\ZZ^{n(n-1)}$ associated to the polynomial ring $\CC[x_{i,j}]$.
The monomial map $\phi_\Lambda$ gives rise to a lattice map $\hat \phi_\Lambda : \ZZ^{\binomial{n}{1} + \dots + \binomial{n}{n-1}} \rightarrow \ZZ^{n(n-1)}$ which sends $\mathcal{P}$ to the polytope $\mathcal{P}_\Lambda$ of the toric variety associated to the monomial map $\phi_\Lambda$.

\begin{definition} \label{def: matching field polytope}
The \textit{matching field polytope} associated to the matching field $\Lambda$ is $\mathcal{P}_\Lambda = (\hat \phi_\Lambda \otimes_\ZZ \RR)(\MP)$.
\end{definition}

The matching field polytope of a partial flag variety is defined in a similar way. In order to distinguish between matching field polytopes for different varieties we write $\mathcal{P}_\Lambda^{{K}}$ for the matching field polytope associated to the partial flag variety $\Fl({K};n)$. {In particular $\mathcal{P}_\Lambda^{k}$ denotes the matching field polytope associated to $\Gr(k,n)$ and $\mathcal{P}_\Lambda^{[n]}$ the matching field polytope associated to $\Fl_n$.}

\begin{example}
Consider the flag variety $\Fl_3$ which embeds into the product $Y_\mathcal{P} = \PP^2\times \PP^2$. Note that $Y_\mathcal{P}$ is a toric variety whose polytope is the product of the two triangles $\mathcal{P} = \Delta_2\times \Delta_2 \subseteq \RR^4$. We define the matrices $\mathbf B$ and $\mathbf D$ and the vector $\rho$ as follows:
\[
\mathbf B = 
\begin{bmatrix}
1 & 0 & 0 & 0\\ 
0 & 1 & 0 & 0\\ 
-1 & -1 & 0 & 0\\ 
0 & 0 & 1 & 0\\ 
0 & 0 & 0 & 1\\ 
0 & 0 & -1 & -1
\end{bmatrix}, \quad
\rho =
\begin{bmatrix}
0 \\ 0 \\ 1 \\ 0 \\ 0 \\ 1
\end{bmatrix} \quad \text{and} \quad
\mathbf D =
\begin{bmatrix}
1 & 1 & 1 & 0 & 0 & 0 \\
0 & 0 & 0& 1& 1 & 1\\
\end{bmatrix}.
\]
The H-description of $\mathcal{P}$ is given by $\{x \in \RR^4 \colon \mathbf B \cdot x \ge \rho \}$. The grading matrix is given by $\mathbf D$ which gives rise to a $\ZZ^2$ grading on $R_\mathcal{P}$.
Let us consider the \textit{diagonal matching field} $\Lambda : \{ I \mid \varnothing \subsetneq I \subsetneq [3] \} \rightarrow S_1 \sqcup S_2$ which sends each subset $I$ to its respective identity permutation. The corresponding monomial map is given by
\[
\phi_\Lambda : \CC[p_1, p_2, p_3, p_{12}, p_{13}, p_{23}] \rightarrow \CC[x_{1,1}, \dots, x_{2,3}], \quad p_i \mapsto x_{1,i}, \quad p_{ij} \mapsto x_{1,i}x_{2,j}.
\]
The kernel of $\phi_\Lambda$ is a principal ideal generated by $p_1p_{23} - p_2p_{13}$. Observe that this binomial is homogeneous with respect to the grading induced by $\mathbf D$, in particular it has degree $(1,1) \in \ZZ^2$.

We identify the polytope $\mathcal{P}$ with its  image under the map $x \mapsto B \cdot x + \rho$ which is given by the convex hull of columns of the matrix
\[
V(\mathcal{P}) = 
\begin{bmatrix}
1 & 1 & 1 & 0 & 0 & 0 & 0 & 0 & 0 \\
0 & 0 & 0 & 1 & 1 & 1 & 0 & 0 & 0 \\
0 & 0 & 0 & 0 & 0 & 0 & 1 & 1 & 1 \\
1 & 0 & 0 & 1 & 0 & 0 & 1 & 0 & 0 \\
0 & 1 & 0 & 0 & 1 & 0 & 0 & 1 & 0 \\
0 & 0 & 1 & 0 & 0 & 1 & 0 & 0 & 1
\end{bmatrix}.
\]
The rows of $V(\mathcal{P})$ are indexed by the variables of the total coordinate ring $p_1$, $p_2$, $p_3$, $p_{12}$, $p_{13}$ and $p_{23}$.
The columns are indexed by pairs of variables of different degrees, namely $(p_i, p_{jk})$ for some $i,j,k, \in [3]$.

Let us consider the map $\hat \phi_\Lambda : \ZZ^6 \rightarrow \ZZ^6$ which corresponds to the monomial map $\phi_\Lambda$. This map sends $\MP$ to the matching field polytope $\mathcal{P}_\Lambda$ and is given by the matrix
\[
\hat \phi_\Lambda = 
\begin{bmatrix}
1 & 0 & 0 & 1 & 1 & 0 \\
0 & 1 & 0 & 0 & 0 & 1 \\
0 & 0 & 1 & 0 & 0 & 0 \\
0 & 0 & 0 & 0 & 0 & 0 \\
0 & 0 & 0 & 1 & 0 & 0 \\
0 & 0 & 0 & 0 & 1 & 1
\end{bmatrix}
\text{ which sends } V(\mathcal{P}) \text { to } \hat\phi_\Lambda(V(\mathcal{P})) = 
\begin{bmatrix}
2 & 2 & 1 & 1 & 1 & 0 & 1 & 1 & 0 \\
0 & 0 & 1 & 1 & 1 & 2 & 0 & 0 & 1 \\
0 & 0 & 0 & 0 & 0 & 0 & 1 & 1 & 1 \\
0 & 0 & 0 & 0 & 0 & 0 & 0 & 0 & 0 \\
1 & 0 & 0 & 1 & 0 & 0 & 1 & 0 & 0 \\
0 & 1 & 1 & 0 & 1 & 1 & 0 & 1 & 1
\end{bmatrix}.
\]
The rows of the above matrices are indexed by the variables $x_{1,1}$, $x_{1,2}$, $x_{1,3}$, $x_{2,1}$, $x_{2,2}$ and $x_{2,3}$ respectively. The matching field polytope $\mathcal{P}_\Lambda$ is the convex hull of the columns of the matrix $\hat\phi_\Lambda(V(\mathcal{P}))$.  We may compute this polytope using the \textit{Polyhedra} package in \textit{Macaulay2} \cite{M2}, which shows that it is a $3$-dimensional polytope with seven vertices, eleven edges and six faces. % 
\end{example}

\subsection{Newton--Okounkov bodies}
Matching field polytopes are examples of Newton-Okounkov bodies. Following the language in \cite{Kaveh2019a}, consider the total coordinate ring $R_\MP=\CC[p_I|\varnothing \subsetneq I\subsetneq [n]]$ of $\PP^{\binomial{n}{1} - 1}\times \dots \times \PP^{\binomial{n}{n-1} - 1}$ where $\binomial{n}{k}=\binom{n}{k}$ for each $k$. Fix a matching field $\Lambda$ which gives rise to a toric degeneration of $\Fl_n$. Recall the lattice map $\hat \phi_\Lambda$ described in the \ref{item: x_phi}rd step of \Cref{proc: toric subvariety}. After an arbitrary choice of linear ordering $\succ$ on $\ZZ S$, we can derive from $\Lambda$ a valuation:
\begin{align*}
    \widetilde{\mathfrak{v}}_\Lambda:\CC[p_I]\setminus \{0\}&\to \ZZ S,\\
    \sum_{\alpha\in \ZZ^{\binomial{n}{1} + \dots + \binomial{n}{n-1}}}c_\alpha z^\alpha &\mapsto \min\limits_{\succ}\{\hat\phi_\Lambda(\alpha)|c_\alpha \neq 0\}.
\end{align*}
In particular, the monomials $p^v$ in $\CC[p_I]$ are sent to $\hat\phi_\Lambda(v)$. So the image of $\widetilde{\mathfrak{v}}_\Lambda$ agrees with the image of $\hat \phi_\Lambda$. Let $A^{[n]}=\CC[p_I]/(\ker \phi)$ be the homogeneous coordinate ring of $\Fl_n$. We obtain a quasi-valuation $\mathfrak{v}_\Lambda$ on $A^{[n]}$ given by: 
\begin{align*}
\mathfrak{v}_\Lambda(\widetilde{f})=\max \{\widetilde{\mathfrak{v}}_\Lambda(f) \mid f\equiv \widetilde{f}\mod\ker \phi_\Lambda\}.
\end{align*} 
Along with $\mathfrak{v}_\Lambda$, comes the notion of a \textit{Khovanskii basis}, which is a set $\mathfrak{B}$ of algebra generators for $A^{[n]}$ so that $\mathfrak{v}_\Lambda(\mathfrak{B})$ generates the \textit{value semigroup} $S(A^{[n]},\mathfrak{v}_\Lambda)=\Ima \mathfrak{v}_\Lambda$. Since $\Lambda$ induces a toric degeneration of $\Fl_n$, it follows that $\mathfrak{v}_\Lambda$ is a valuation and (the image under the quotient map of) the Pl\"ucker variables form a Khovanskii basis for $A^{[n]}$. See \cite[Corollary~2]{bossinger2021full}. 
Furthermore, $\mathcal{P}^{[n]}_\Lambda$ is the \textit{Newton-Okounkov body} $\Delta(A^{[n]},\mathfrak{v}_\Lambda)$ of the pair $(A^{[n]},\mathfrak{v}_\Lambda)$. By \cite[Corollary~1]{bossinger2021full}, we have that $\Delta(A^{[n]},\mathfrak{v}_\Lambda)$ is a Minkowski sum of polytopes. Moreover, by \Cref{Minkowski}, we have that $\Delta(A^{[n]},\mathfrak{v}_\Lambda)$ is the Minkowski sum of the polytopes $\mathcal{P}^{k}_\Lambda$, which are themselves Newton-Okounkov bodies for the Grassmannian  $\mathcal{P}^{k}_\Lambda = \Delta(A^{k},\mathfrak{v}_\Lambda)$, where $A^{k}$ denotes the homogeneous coordinate ring of the Grassmannian $\mathrm{Gr}(k,n)$.

\subsection{Minkowski sums of matching field polytopes}\label{Mink}

Here, we show that a matching field polytope for $\Fl_n$ can be constructed as a Minkowski sum of matching field polytopes for the Grassmannians $\Gr(k,n)$ for all $k \in [n-1]$. Similarly to \Cref{recipe}, we work with the full flag variety for simplicity, however all constructions hold for partial flag varieties $\Fl({K};n)$.

\begin{notation}
The Pl\"ucker embedding of the flag variety, $\Fl_n\to \PP^{\binomial{n}{1} - 1}\times \dots \times \PP^{\binomial{n}{n-1}-1}$ factors through the natural embedding into a product of Grassmannians $\Fl_n\to \mathrm{Gr}(1,n)\times \dots \times \mathrm{Gr}(n-1,n)$. For each subset ${K} = \{k_1, \dots, k_t\} \subseteq [n-1]$ we write $\phi^{{K}}$ for the restriction of the map $\phi$ from \Cref{eq: flag plucker ring map}~to~the~subring 
\[
R_{\mathcal{P}^{{K}}} = 
R_{\Delta_{\binomial{n}{k_1} - 1}} \otimes_\CC \dots \otimes_\CC R_{\Delta_{\binomial{n}{k_t} - 1}} = \CC\left[p_I \mid I \subseteq [n], \ |I| \in {K}\right]. 
\]
The ring map $\phi^{{K}}$ defines the Pl\"ucker embedding of the partial flag variety $\Fl({K};n)$ into $\PP^{\binomial{n}{k_1} - 1} \times \dots \times \PP^{\binomial{n}{k_t} - 1}$. Similarly, for each matching field $\Lambda$ for $\Fl_n$, we define the monomial map $\phi^{{K}}_{\Lambda}$ to be the restriction of $\phi_\Lambda$ from \Cref{eq: matching field monomial map} to $R_{\mathcal{P}^{{K}}}$. We note that the construction of the matching field polytope for a full flag variety in \Cref{recipe} naturally restricts to the setting of partial flag varieties, where the polytope is
\[
\mathcal{P}_{\Lambda}^{{K}} = (\hat \phi_\Lambda^{{K}} \otimes_\ZZ \RR) (\mathcal{P}^{{K}}).
\]

In particular, if ${K} = \{k\}$ is a singleton then we write $\mathcal{P}_\Lambda^{k}$ for the matching field polytope associated to the Grassmannian matching field. For further treatment of matching field polytopes for the Grassmannian we refer the reader to \cite{clarke2022combinatorial}. In particular, {we recall} that the vertices of $\mathcal{P}_\Lambda^{k}$ have an explicit description given by
\[
v_{I, \Lambda} := 
\hat\phi_\Lambda^{k} (p_I) = 
\sum_{i = 1}^k e_{i,j_{\Lambda{(I)}(i)}} \in 
\ZZ^{k \times n}
\text{ for each } I = \{j_1 < \dots < j_k\} \in \binom{[n]}{k}.
\]

Throughout the rest of the paper, we will use the superscripts $[n]$, $k$ and ${K}$ to distinguish our matching fields, polytopes and maps{, respectively for $\Fl_n, \Gr(k,n)$ and $\Fl({K};n)$}. In particular, given a matching field $\Lambda$ for $\Fl_n$, we write $\Lambda^{{K}}$ for the restriction of $\Lambda$ to the subsets $\{I \subseteq [n] \mid |I| \in {K} \}$.
However, if the context is clear then we will remove the superscripts for ease of notation.
\end{notation}

We are now ready to prove the Minkowski sum property of toric subvarieties of (partial) flag varieties. 

\begin{proof}[{\bf Proof of \Cref{Minkowski}}]
Observe that the polytope $\mathcal{P}=\Delta_{\binomial{n}{1} - 1}\times \dots \times \Delta_{\binomial{n}{n-1} - 1} \subseteq \RR^{\binomial{n}{1} + \dots + \binomial{n}{n-1}}$ is a Minkowski sum of simplices 
\[
\mathcal{P} = \sum_{k = 1}^{n-1} 
\{0\}\times \dots \times \{0\} \times \Delta_{\binomial{n}{k} - 1}\times \{0\} \times \dots \times \{0\}.
\] 
Since linear maps commute with taking Minkowski sums, we have that the matching field polytope $\mathcal{P}_\Lambda^{[n]} = (\hat\phi_\Lambda \otimes_\ZZ \RR)(\mathcal{P})$ is the Minkowski sum of polytopes $(\hat\phi_\Lambda \otimes_\ZZ \RR)(\{0\}\times \dots \times \Delta_{\binomial{n}{k} - 1}\times \dots \times \{0\})$ for each $k \in [n-1]$. By definition, each summand is the matching field polytope $\mathcal{P}_\Lambda^{k} \subseteq \RR^{(n-1)\times n}$ associated to the Grassmannian embedded into a product of projective spaces $\Gr(k,n) \subseteq \PP^{\binomial{n}{1} - 1} \times \dots \times \PP^{\binomial{n}{n - 1} - 1}$.
\end{proof}

\begin{remark}\label{rmk: general minkowskii sum lambda P}
We note that \Cref{Minkowski} may be generalised as follows. For each $\lambda \in \NN^{n-1}$ define the polytope $\lambda \MP := \lambda_1\Delta_{\binomial{n}{1} - 1}\times \dots \times \lambda_{n-1}\Delta_{\binomial{n}{n-1} - 1}$. If $\lambda_i > 0$ for all $i$ then $\lambda\MP$ corresponds to a choice of a very ample line bundle on $\PP^{\binomial{n}{1} - 1}\times\dots\times \PP^{\binomial{n}{n-1} - 1}$ which defines a different embedding into projective space. Otherwise $\lambda\MP$ corresponds to a choice of a very ample line bundle on the subspace $\PP^{\binomial{n}{k_1} - 1}\times \dots \times \PP^{\binomial{n}{k_t} - 1}\subseteq \PP^{\binomial{n}{1} - 1}\times\dots\times \PP^{\binomial{n}{n-1} - 1}$, where ${K} = \{k_1<\dots<k_t\} \subseteq [n-1]$ indexes the nonzero entries of $\lambda$. 

Following \Cref{proc: toric subvariety}, the H-description of each scaled simplex $\lambda_k \Delta_{\binomial{n}{k} - 1}$ is given by $\{x \in \RR^{\binomial{n}{k} - 1} \mid \mathbf J_{\binomial{n}{k} - 1} \cdot x \ge -\rho_{\binomial{n}{k}}' \}$ where $\rho_{\binomial{n}{k}}' = (0, \dots, 0, \lambda_k)^T \in \ZZ^{\binomial{n}{k}}$. So, the H-description of $\lambda \mathcal{P}$ is given by the same matrix $\mathbf B$; hence the grading matrix is also given by the same matrix $\mathbf D$. Similarly to \Cref{recipe}, define $\rho' = [\rho_{\binomial{n}{1}}'^T | \dots | \rho_{\binomial{n}{n-1}}'^T]^T$. Notice that $\mathbf D$ sends the vector $\rho'$ to $\lambda$. 

The polytope $\lambda \mathcal{P}_\Lambda := (\hat\phi_\Lambda \otimes_\ZZ \RR)(\lambda\MP)$ is different from the matching field polytope $\mathcal{P}_\Lambda$. However the proof of \Cref{Minkowski} may be applied to show that $\lambda \mathcal{P}_\Lambda$ is a Minkowski sum of polytopes associated to Grassmannians. In this case, the summands are given by $(\hat\phi_{\Lambda} \otimes_\ZZ \RR)(\lambda_k \Delta_{\binomial{n}{k} - 1})$. In particular, if $\lambda$ is the indicator vector for the subset ${K} \subseteq [n-1]$ then we have shown that the matching field polytope $\mathcal{P}^{{K}}_\Lambda$ for the partial flag variety $\Fl({K};n)$ is the Minkowski sum of the polytopes $\mathcal{P}^{k}_\Lambda$ where $k \in {K}$.
\end{remark}

\section{Toric degenerations from matching fields}\label{sec: toric degen from matching fields}
In this section, we start by recalling the definition of tropicalizations as introduced in \cite[Section 3.2]{maclagan2021introduction}.
We also recall how a matching field gives rise to a toric degeneration of a (partial) flag variety. We then show that all matching field polytopes are normal, followed by a proof of \Cref{thm:toric degen partial flag}. 

Fix a {Laurent polynomial ring $R = \CC[x_1^{\pm 1}, \dots, x_n^{\pm 1}]$} on $n$ variables together with a \textit{weight vector} $w \in \RR^n$. For any Laurent polynomial $f = \sum_{{\bm u} \in \ZZ^n} a_{\bm u} \mathbf{x}^{\bm u}$, its \textit{initial form} with respect to $w$ is defined to be $\mathrm{in}_w(f) = \sum_{\bm u} a_{\bm u}\mathbf{x}^{\bm u}$ where the sum is taken over all ${\bm u} \in \ZZ^n$ such that ${\bm u} \cdot w$ has the minimum value. Given an ideal $I \subseteq R$, we define the \textit{initial ideal} with respect to $w$ to be $\mathrm{in}_{w}(I) = \langle \mathrm{in}_w(f) \mid f \in I \rangle$, that is the ideal generated by the initial terms of all polynomials in $I$. The \textit{tropicalization}  {(as a set)} is defined as $\Trop(I) = \{ w \in \RR^n \mid \mathrm{in}_{w}(I) \text{ contains no monomials}\}$. It inherits the structure of a polyhedral fan from the \textit{Gr\"obner fan} of $I$, which is defined as the complete polyhedral fan in $\RR^n$ where two weight vectors $w$ and $w'$ lie in the relative interior of the same cone of the Gr\"obner fan if and only if $\mathrm{in}_{w}(I) = \mathrm{in}_{w'}(I)$.
In particular, in order to take the tropicalization of the partial flag variety $\Fl(K;n)$, we need to consider the extension of the Pl\"ucker ideal to the Laurent polynomial ring. The corresponding variety is the intersection of the partial flag variety with the algebraic torus, which will be denoted by $\Fl(K;n)^\circ$.

We restrict ourselves to the complete flag variety $\Fl_n$ and note that all arguments apply to partial flag varieties and, in particular, Grassmannians. Consider the polynomial ring $\CC[\mathbf{p}]=\CC[p_I \mid I \subseteq [n]]$ and the Pl\"{u}cker ideal $\MI=\Ker(\phi)$ that defines $\Fl_n$ as a subvariety of multiprojective space. We mostly refer to the approach described in \cite{Miller2004, maclagan2021introduction}. Fix a coherent matching field $\Lambda$ for the flag variety $\Fl_n$ which is induced by the weight matrix $M_\Lambda$. The \textit{weight vector} $w_{\Lambda}\in \RR^{\binomial{n}{1} + \dots + \binomial{n}{n-1}}$ associated to $\Lambda$ and $M_\Lambda$ is given by
\[w_{\Lambda} =\left(\min\limits_{\sigma\in S_n}\left\{\sum_{a=1}^{|I|} (M_\Lambda)_{a,i_{\sigma(a)}}\right\}\right)_{\varnothing \subsetneq I\subsetneq [n]},
\]
{where $I=\{i_1< \dots < i_{|I|}\}$.} 

We use Gr\"obner degeneration to construct a flat family $\pi: \mathcal{F}\to \mathbb{A}^1$ over $\mathbb{A}^1$ such that the initial ideal $\mathrm{in}_{w_{\Lambda}}(\MI)$ is associated to the fiber over $\{0\}$. Explicitly, for each $t \in \mathbb A^1 \backslash \{0\}$ the fiber $\pi^{-1}(t)$ is the subvariety of multiprojective space defined by the ideal 
\[
\Bar{\MI_t}=
\left\langle 
t^{-\min_{\bm{u}}\{w_{\Lambda}\cdot \bm{u}\}}
f(t^{w_{\Lambda \cdot I_1}}p_{I_1},\dots,t^{w_{\Lambda \cdot I_d}}p_{I_d})
\mid
f=\sum a_{\bm u}\mathbf{p}^{\bm{u}}\in \MI
\right\rangle,
\]
where $d= \binomial{n}{1}+\dots+\binomial{n}{n-1}$ is the number of Pl\"ucker coordinates in $\CC[\mathbf{p}]$. 
If $\mathrm{in}_{w_{\Lambda}}(\MI)$ is toric, then it follows that $\Lambda$ produces a toric degeneration of $\Fl_n$. The question of whether $\mathrm{in}_{w_{\Lambda}}(\MI)$ is toric is treated in \cite[Chapter 11]{Sturmfels1996}. By \cite[Theorem~11.3]{Sturmfels1996}, we always have the inclusion 
\begin{equation}\label{eq: inclusion initial ideal}
    \mathrm{in}_{w_{\Lambda}}(\MI)\subseteq \MI_\Lambda
\end{equation} where $\MI_\Lambda$ is the toric ideal $\Ker(\phi_\Lambda)$. On the other hand, by \cite[Theorem 11.4]{Sturmfels1996}, the ideal $\mathrm{in}_{w_\Lambda}$ is toric if and only if the reverse inclusion holds.

\begin{remark}
    Suppose that $M$ and $M'$ are two weight matrices that induce the same matching field $\Lambda$. Equivalently, the initial term of each Pl\"ucker form with respect to either $M$ or $M'$ is the same. Note that every matching field that we consider throughout this paper produces a toric degeneration. Therefore $\mathrm{in}_{w}(\MI) = \MI_\Lambda = \mathrm{in}_{w'}(\MI)$, where $w$ and $w'$ are the weight vectors associated to $M$ and $M'$. In particular, $w$ and $w'$ lie in the same cone of the Gr\"obner fan of the Pl\"ucker ideal $\MI$. So, given a coherent matching field $\Lambda$, by a slight abuse of notation we write $M_\Lambda$ for a weight matrix that induces $\Lambda$ and we write $w_\Lambda$ for the weight vector induced by $M_\Lambda$. If a matching field does not give rise to a toric degeneration then we specify $M_\Lambda$ and $w_\Lambda$ explicitly.
\end{remark}

\begin{example} \label{Example: different weight matrices same toric degeneration}
    We give an example of two different weight matrices that induce the same matching field, hence they give same toric degeneration of $\Gr(2, 4)$. Note that for any weight matrix $M \in \RR^{2 \times 4}$, the matching field induced by $M$ is invariant under adding a constant to any column of $M$. So, we may always assume that the first row of $M$ is all zeros. Let $M_1$ and $M_2$ be as follows:
    \[
    M_1 = 
    \begin{bmatrix}
        0 & 0 & 0 & 0 \\
        3 & 2 & 1 & 0 
    \end{bmatrix}
    \quad \text{and} \quad
    M_2 = 
    \begin{bmatrix}
        0 & 0 & 0 & 0 \\
        9 & 5 & 4 & 0 
    \end{bmatrix}.
    \]
    In this case, we note that the induced matching field is well-defined because the entries in the second row of the weight matrices are distinct. The induced weights on the Pl\"ucker variables $p_{12}$, $p_{13}$, $p_{23}$, $p_{14}$, $p_{24}$, $p_{34}$ are given by
    \[
    w_1 = (2, 1, 1, 0, 0, 0) 
    \quad \text{and} \quad
    w_2 = (5, 4, 4, 0, 0, 0)
    \]
    respectively. The corresponding initial ideals %of the Pl\"ucker ideal 
    are: $\mathrm{in}_{w_1}(\MI) = \langle p_{13}p_{24} - p_{14}p_{23} \rangle = \mathrm{in}_{w_2}(\MI)$.
\end{example}

To prove the reverse inclusion {to (\ref{eq: inclusion initial ideal}),} we follow a similar strategy to \cite{clarke2022combinatorial} and consider the Hilbert function of $\mathrm{in}_{w_{\Lambda}}(\MI)$ and the Ehrhart function of the matching field polytope.

\begin{definition}\label{def: Ehrhart and normal}
Let $\mathcal{P}\subseteq M_\RR$ a lattice polytope. { We recall the following definitions: }
\begin{itemize}
    \item The \textit{Ehrhart function} $E_\mathcal{P}\colon\NN\to\NN$ sends $m\in \NN$ to the number of lattice points of the $m$th dilate $m\mathcal{P}$ of $\mathcal{P}$; that is $E_\mathcal{P}(m) = |m\mathcal{P} \cap M|$.
    \item  The lattice polytope $\mathcal{P}$ is \textit{normal} if for all $m\in \NN$ each lattice point of $m\mathcal{P}\cap M$ is a sum of $m$ lattice points of $\mathcal{P}$. That is, for all $\alpha\in m\mathcal{P}\cap M$ there exist $\alpha_1, \dots, \alpha_m \in \mathcal{P} \cap M$ such that $\alpha=\alpha_1+\dots+\alpha_m$.
\end{itemize}
\end{definition}

We will now show that all coherent matching field polytopes for all partial flag varieties are normal. Recall from \Cref{rmk: general minkowskii sum lambda P} that the polytope $\lambda \mathcal{P}$ is a product of dilated simplicies for each each $\lambda \in \NN^{n-1}$.

\begin{restatable}{prop}{latticeptsofflagpolytopes}
\label{prop: lattice points of flag polytopes}
Fix ${\lambda}=(\lambda_1,\dots,\lambda_{n-1})\in \NN^{n-1}$ and consider the polytope ${\lambda}\mathcal{P}_\Lambda:=(\hat{\phi}_\Lambda \otimes_\ZZ \RR) ({\lambda} \mathcal{P})$. Then any lattice point $\alpha\in {\lambda} \mathcal{P}_\Lambda$ can be written as 
\[
\alpha=\sum_{k=1}^{n-1}\sum_{j=1}^{\lambda_k} v_{{I_{(k,j)}},\Lambda}
\]
for some vertices $v_{{I_{k,j}},\Lambda}$ of $\mathcal{P}^{k}_\Lambda$. In particular, for any $\ell\in \NN$ and ${K= \{ k_1 < \dots <  k_\ell\}}$ the matching field polytopes $\mathcal{P}_\Lambda^{{K}}$ is normal, when considered as a lattice polytope with respect to the lattice $\ZZ S$.
\end{restatable}

\begin{proof}
Fix a lattice point $\alpha \in \lambda \mathcal{P}_\Lambda$. Recall that $\ZZ S$ is generated by $\{s_I := \hat{\phi}_\Lambda(e_I) \mid \varnothing \subsetneq I \subsetneq [n]\}$ of where $e_I \in \ZZ^{\binomial{n}{1} + \dots + \binomial{n}{n-1}}$ is a standard lattice generator. Each lattice point $s_I$ is a vertex of the polytope $\hat{\phi}_\Lambda(\Delta_k) = \mathcal{P}_\Lambda^{k}$ where $k = |I|$. Therefore, we can write $\alpha$ as a sum of lattice generators $\alpha=\sum_{I} a_I s_I $ for some $a_I \in \NN$.
{By \Cref{lemma: grading}, we have} that $\phi_\Lambda$ is {concurrent} with respect to the {$\Cl(Y_\MP)$}-grading of $R_{\mathcal{P}}$. That is, if $\hat{\phi}$ is the map defined in \cref{eq:phi_hat_map}, $\alpha=\hat{\phi}_\Lambda (\beta)$ for some lattice point $\beta$ { and we write ${\lambda = (\lambda_1, \dots, \lambda_{n-1})}$ for the $\Cl(Y_\MP)$-degree of corresponding monomial $x^\alpha \in \CC[x_{i,j}]$}.
For each lattice generator $s_I = \hat{\phi}_\Lambda(e_I)$, we have that {the corresponding element  $p_I$ has degree} $(0, \dots, 0, 1, 0, \dots, 0)$ where the $1$ appears in position $k = |I|$. It follows that $\lambda_k = \sum_{|I| = k} a_I$ and so we may rewrite the above sum so that $\alpha = \sum_{k = 1}^{n-1} \sum_{j = 1}^{\lambda_k} v_{I_{(k, j)}, \Lambda}$ where $I_{(k, j)} \subseteq [n]$ is a $k$-subset.

Fix ${K=\{k_1<\dots<k_\ell\}}$. The polytope $\mathcal{P} := \mathcal{P}^{{K}}_\Lambda$ is $(\hat{\phi}_\Lambda \otimes_\ZZ \RR) ({\lambda}\MP)$, where $\lambda$ is the indicator vector of ${K}$. That is, for all $i \in [n-1]$ we have ${\lambda_i }= 1 $ if $i \in {K}$, otherwise $\lambda_i = 0$. To prove that $\mathcal{P}$ is normal let $\alpha \in m\mathcal{P} \cap \ZZ S$ be any lattice point. By linearity, $\alpha$ is a lattice point of the polytope $(\hat\phi_\Lambda \otimes \ZZ)(\lambda' \MP)$ where $\lambda' = m\lambda$. By the previous part, we have that $\alpha = \sum_{k \in {K}} \sum_{i = 1}^{\lambda'_k} v_{I_{(k,j)}, \Lambda}$, where $v_{I_{(k,j)}, \Lambda}$ is a vertex of $\mathcal{P}_\Lambda^{k}$. For each $j \in [m]$ define $\alpha_j = \sum_{k \in {K}} v_{I_{(k, j)}, \Lambda} \in \mathcal{P} \cap \ZZ S$. Therefore $\alpha = \alpha_1 + \dots + \alpha_m$ and so $\mathcal{P}$ is normal.
\end{proof}

Before proving \Cref{thm:toric degen partial flag}, we recall the definition of combinatorial mutations via tropical maps~\cite{clarke2022combinatorial}. 

\smallskip
Let $M$ be a lattice and let $N = M^*$ be its dual together with the corresponding $\RR$-vector spaces $M_\RR = M \otimes_\ZZ \RR$ and $N_\RR = N \otimes_\ZZ \RR$. We write $\langle \cdot , \cdot \rangle : M_\RR \times N_\RR \rightarrow \RR$ for the natural pairing between $M_\RR$ and $N_\RR$. Let $w \in M$ be a primitive lattice point and  $F\subset w^\perp \subseteq N_\RR$ a lattice polytope. 
\begin{definition} \label{def : combinatorial mutation}
The \emph{tropical map} associated to $w$ and $F$ is the piecewise linear map defined as
\[
    \varphi_{w,F}: M_\RR \to M_\RR, \quad 
                    x\mapsto x-\min\{ \langle x,v \rangle \mid v\in F \} w.
\]
If $\mathcal{P}\subset M_\RR$ is a lattice polytope such that $\varphi_{w,F}(\mathcal{P})$ is convex, then the polytope $\varphi_{w,F}(\mathcal{P})$ is called a \emph{combinatorial mutation} of $\mathcal{P}$. If a polytope $\mathcal{P}'\subseteq M_\RR$ can be obtained from $\mathcal{P}$ by a sequence of combinatorial mutations, then $\mathcal{P}'$ and $\mathcal{P}$ are called \emph{combinatorial mutation equivalent}.
\end{definition}

Note that tropical maps are piecewise unimodular maps which act bijectively on the lattice $M_\RR$. Many properties of polytopes are preserved by combinatorial mutation, such as their Ehrhart polynomial~\cite{Akhtar2012}.

\begin{remark}\label{rmk : inverse mutation} 
Combinatorial mutations were originally defined for lattice points in the dual lattice $N_\RR$ \cite[Definition 5]{Akhtar2012}. In the literature \cite{clarke2022combinatorial}, it is often required that lattice polytopes $\mathcal{P}\subseteq M_\RR$ that are subject to combinatorial mutations $\mathcal{P} \mapsto \mathcal{P}'$ must contain the origin. In this case, the dual polytope polytope $\mathcal{P}^*$ is well defined and the change $(\mathcal{P}^*) \mapsto (\mathcal{P}'^*)$ coincides with the definition of a combinatorial mutation above given by tropical maps. Note that composing $\varphi_{w,F}$ with a translation $x\mapsto x+u$ where $u \in M$ is a lattice point, is equal to the tropical map $\varphi_{w,F+u}$. So if two polytopes $\mathcal{P}$ and $\varphi_{w,F}(\mathcal{P})$ share at least one point, then we may assume without loss of generality that this point in the origin. We note that translation makes no difference to the construction of toric varieties from polytopes. We also note that for all pairs of matching field polytopes that we show to be related by combinatorial mutations, these pairs of polytopes always share at least one lattice point. Therefore, we may safely drop the assumption that the polytopes contain the origin.
\end{remark}

We are now ready to give a proof of \Cref{thm:toric degen partial flag}. 
\begin{proof}[{\bf Proof of \Cref{thm:toric degen partial flag}}]
We prove the result for the complete flag variety $\Fl_n$ and we note that the general case is completely analogous. 
Recall that the diagonal matching field $B^{w_0}$ induces a toric degeneration of the flag variety which is associated to the Gelfand-Tsetlin polytope. Therefore, the Hilbert function of $\Fl_n$ agrees with the Hilbert function of $\CC[p_I] / \MI_{B^{w_0}}$. Since $\mathcal{P}_{w_0}$ is normal by \Cref{prop: lattice points of flag polytopes}, we have that the Hilbert function of $\CC[p_I] / \MI_{B^{w_0}}=\CC[\mathrm{Cone}(\mathcal{P}_{w_0})\cap \ZZ S\times \ZZ]$ is equal to the Ehrhart polynomial of $\mathcal{P}_{w_0}$. By \cite{Akhtar2012}, tropical maps preserve the Ehrhart function and so the Ehrhart polynomial of $\mathcal{P}_{\Lambda}$ coincides with the Hilbert function of $\Fl_n$. Moreover, $\mathcal{P}_{\Lambda}$ is normal by \Cref{prop: lattice points of flag polytopes}, so we have that the Hilbert function of $\CC[p_I] / \MI_{\Lambda}=\CC[\mathrm{Cone}(\mathcal{P}_{\Lambda})\cap \ZZ S\times \ZZ]$ is equal to the Ehrhart polynomial of $\mathcal{P}_{\Lambda}$. Hence, the Hilbert function of $\Fl_n$ and the Hilbert function of $\CC[p_I] / \MI_{\Lambda}$ agree. Since Gr\"obner degenerations preserve the Hilbert function, 
$\CC[p_I] / {\rm in}_{w_{\Lambda}}(\MI)$ and $\CC[p_I] / \MI_{\Lambda}$ have the same Hilbert function. Since ${\rm in}_{w_{\Lambda}}(\MI) \subseteq \MI_{\Lambda}$, it follows that the ideals are equal. In particular, ${\rm in}_{w_{\Lambda}}(\MI)$ is a toric ideal, as desired. 
\end{proof}

We finish this section by showing that the set of Pl\"ucker variables forms a SAGBI basis for the Pl\"ucker algebra. We first recall the definition of {\it SAGBI basis} from \cite{robbiano1990subalgebra} in our setting. Note that the notion of Khovanskii bases, from \cite{Kaveh2019a}, generalizes the notion of a SAGBI bases to non-polynomial algebras.

\begin{definition}\label{sagbi}
Fix a natural number $n$ and a subset ${K} \subseteq [n]$. Let $\Lambda$ be a matching field for the partial flag variety $\Fl({K};n)$.  Let $\mathcal{A}({K};n)$ be the Pl\"ucker algebra of $\Fl({K};n)$ and let $A_{\Lambda} =\CC[{\rm in}_{w_{\Lambda}}(p_I)]$ be the algebra of $\Lambda$. The set of Pl\"ucker forms $\{p_I\}\subset \CC[x_{i,j}]$ is a SAGBI basis for $\mathcal{A}({K};n)$ with respect to the weight vector $w_{\Lambda}$ if and only if for each $I$, the initial form ${\rm in}_{w_{\Lambda}}(p_I)$ is a monomial and
${\rm in}_{w_{\Lambda}}(\mathcal{A}({K};n))=A_\Lambda$. Here, $w_{\Lambda}$ is the weight vector induced by the matrix $M_\Lambda$.
\end{definition}

Obtaining a toric degeneration of $\Fl({K};n)$  is equivalent to obtaining a SAGBI basis for the Pl\"ucker algebra $\mathcal{A}({K};n)$, since the set of Pl\"ucker forms $\{p_I\}\subset \CC[x_{i,j}]$ is a SAGBI basis for the Pl\"ucker algebra with respect to a matching field if and only if the ideals ${\rm in}_{w_{\Lambda}}(\MI)$ and $\MI_{\Lambda}$ are equal. See \cite[Theorem 11.4]{Sturmfels1996} or \cite[Theorem 11.4]{Mohammadi2019}. Hence, as an immediate corollary of  \Cref{thm:toric degen partial flag} we have the following: 
\begin{corollary}\label{cor:SAGBI}
The Pl\"ucker variables $p_I$ form a SAGBI basis for the Pl\"ucker algebra with respect to the weight vectors arising from matching fields.
\end{corollary}

\begin{remark}
It is proven in \cite{clarke2020toric} that the ideal $\MI_{\Lambda}$ is quadratically generated for any so-called block diagonal matching field. Our computations show that this is true for every matching field of the form $B^\sigma$ up to $n=8$. We expect this to be true for all matching fields $B^\sigma$. Notice that, providing a bound for the degree of the generators of a general ideal is a difficult problem. Such problem is usually studied for special families of combinatorial ideals; see, e.g., \cite{White, hibi1987distributive, Ohsugi, ene2011monomial, Cone, mohammadi2010weakly}. 
\end{remark}

\section{Mutations of matching field polytope of partial flag varieties}\label{mutation section}

In this section, we introduce the matching fields $\Bsig$ which are a family of coherent matching fields for all partial flag varieties $\Fl({K};n)$. Each such matching field is indexed by a permutation in $\sigma \in S_n$, generalizing families of coherent matching fields considered in \cite{Mohammadi2019} and \cite{clarke2022combinatorial}. We will show that certain matching field polytopes are combinatorial mutation equivalent. This, together with \Cref{thm:toric degen partial flag} implies that such matching fields gives rise to a toric degeneration of $\Fl({K};n)$.

\subsection{Matching fields from permutations}\label{sec: matching fields from permutations}

We begin with the definition of a matching field for the Grassmannian $\Gr(k, n)$.

\begin{definition} \label{def:matching field B sigma}
Let $k>0$ and $\sigma\in S_n$. Given $I= \{p < q < i_3 < \dots < i_k\} \in \binom{[n]}{k}$, the matching field $\Bsig:\binom{[n]}{k}\to S_k$ is defined by
\[
\Bsig(I) = \left\{
\begin{array}{ll}
    id   & \textrm{if }\sigma(p) > \sigma(q) \textrm{ or } |I| = 1,\\
    (12) & \textrm{otherwise.}
\end{array}
\right.
\]
\end{definition}
\begin{lemma} \label{lemma: matching field weight matrix}
All such matching fields are coherent. In particular, $\Bsig$ is induced by the weight matrix 
\begin{align*}
M^\sigma = 
 \begin{pmatrix}
0 & 0 & \dots & 0\\ 
\sigma(1) & \sigma(2)& \dots & \sigma(n)\\ 
 Nn& N(n-1) &\dots & N\\ 
N^2n& N^2(n-1) &\dots & N^2\\ 
 \vdots & \vdots & \vdots & \vdots\\
N^{n-2}n & N^{n-2}(n-1)& \dots & N^{n-2}\\
\end{pmatrix}   
\end{align*}
for any $N \geq n+1.$
\end{lemma}
\begin{proof}
For $k=1$, there is nothing to show. If $k=2$, then let $I = \{p < q\} \in \binom{[n]}{2}$ be any $2$-subset. 

Let $\Lambda$ be the matching field induced by the matrix $M^\sigma$. We have the following:
\[
\Lambda(I) = 
\begin{cases}
id & \text{if } M^\sigma_{1,p} + M^\sigma_{2,q} < M^\sigma_{2,p} + M^\sigma_{1,q},\\
(12) & \text{if } M^\sigma_{2,p} + M^\sigma_{1,q} < M^\sigma_{1,p} + M^\sigma_{2,q}
\end{cases}\ 
= 
\begin{cases}
id & \text{if } \sigma(q) < \sigma(p),\\
(12) & \text{if } \sigma(p) < \sigma(q)
\end{cases}\ 
= B^\sigma(I).
\]
And so the induced matching field coincides with $\Bsig$.
For $k>2$, the result follows directly from the proof of \cite[Proposition~1]{clarke2022combinatorial}.  We sketch the argument below.

Let $\Lambda$ be the matching field induced by $M^\sigma$ and $I \subseteq [n]$ be a subset such that $k := |I|$. We show $\Lambda(I) = B^\sigma(I)$ by induction on $k$. If $k \in \{1, 2\}$, then we are done by the above paragraph. So we may take $k \ge 3$. First, we deduce that $\Lambda(I)(k) = k$ by calculating the induced weight $(w_\Lambda)_I$. Next, we apply induction by considering the set $I \backslash \{\max(I)\}$.
\end{proof}

\begin{example}[Block diagonal matching fields \cite{Mohammadi2019}]
Fix a tuple $\bm{a}=(a_1,\dots,a_r)$ such that $\sum_{i=1}^r a_i = n$. Let $\alpha_0=1$ and $\alpha_\ell=\sum_{i=1}^\ell a_i$ for each $1\leq \ell\leq n$. Consider the sets $I_\ell=\{\alpha_{\ell-1}+1,\dots,\alpha_{\ell}\}$. Define the matching field $B_{\bm{a}}$ such that for each $I\in \binom{[n]}{k}$,
\begin{equation}
    B_{\bm{a}}(I)= \left\{
\begin{array}{ll}
    id   & \textrm{if } |I_t\cap \{1,\dots,k\}| \geq 1, \textrm{ where $I_t$ is the minimal $t$ so that $I_t\cap I\neq \varnothing$}\\
    (12) & \textrm{otherwise.}
\end{array}
\right.
\end{equation}
Note that $B_{\bm{a}}$ is coherent because it coincides with the matching field $\Bsig$ where
$w_{\bm{a}} = (I_1^\downarrow, I_2^\downarrow, \dots I_r^\downarrow) \in S_n$. The notation $I_j^\downarrow$ denotes the sequence whose entries are $I_j$ written in descending order. 
\end{example}

Fix a partial flag variety $\Fl({K};n)$ with ${K} \subseteq [n]$. The matching field $\Bsig$ is defined simultaneously for $\Gr(k, n)$ for all $k \in {K}$. Therefore, $\Bsig$ is a matching field for the partial flag variety $\Fl({K};n)$. Moreover, $B^\sigma$ is coherent by \Cref{lemma: matching field weight matrix}, since the topmost rows of $M^\sigma$ are equal for each $k \in {K}$.

\subsection{Mutations of polytopes of flag varieties}\label{flag variety section}

Here, we will exploit \Cref{Minkowski} to construct toric degenerations of $\Fl({K};n)$. We prove \Cref{thm: DSR_2 Flag varieties} by constructing a sequence of combinatorial mutations between matching field polytopes. Throughout this section, unless otherwise specified, we fix the following setup.

\begin{setup}\label{setup flag}
Fix $n$, a permutation $\sigma \in S_n$ and a positive integer $\ell \in [n-1]$. We define $\lambda = \sigma^{-1}(\ell)$ and $\mu = \sigma^{-1}(\ell+1)$. We assume that $\lambda < \mu$ and define $\tau = (\ell \ \ell+1) \sigma$. For each $k \in [n]$, the matching field polytope for the Grassmannian $\Gr(k,n)$ associated to the matching fields $B^\sigma$ and $B^\tau$ will be denoted by $\mathcal{P}_\sigma^k$ and $\mathcal{P}_\tau^k$ respectively. Each such polytope is assumed to live in $\RR^{k \times n} \subseteq \RR^{(n-1)\times n}$. Explicitly, this inclusion is given by extending the standard basis $\{e_{i,j} \mid i \in [k], \ j \in [n]\}$ of $\RR^{k \times n}$ to the standard basis of $\{e_{i,j} \mid i \in [n-1], \ j \in [n]\}$ of $\RR^{(n-1) \times n}$. From \Cref{Minkowski}, it follows that for each non-empty subset ${K} \subseteq [n-1]$, the matching field polytope for the partial flag variety $\Fl({K};n)$ associated to the matching field $B^\gamma$ is given by $\mathcal{P}_\gamma = \sum_{k \in {K}} \mathcal{P}_\gamma^k$ for each $\gamma \in \{\sigma, \tau\}$. We primarily work with the polytopes $\mathcal{P}_\sigma^k$ and so write $v_I^k$ for the vertex of the polytope $\mathcal{P}_\sigma^k$ corresponding to the $k$-subset $I \subseteq [n]$.
\end{setup}

Note that if any of $\ell$, $\lambda$ or $\mu$ are fixed then the rest are also determined uniquely by $\sigma$.

\begin{definition}\label{def flag tropical map}
Assume \Cref{setup flag} holds. We define $w \in \RR^{(n-1) \times n}$ by
\[ w_{i,j} = 
\begin{cases} 
1 & \text{ if } (i=1 \text{ and } j=\lambda) \text{ or } (i=2 \text{ and } j = \mu),\\
-1 & \text{ if } (i=1 \text{ and } j = \mu) \text{ or } (i=2 \text{ and } j = \lambda),\\
0 & \text{ otherwise}.
\end{cases}\]
We define $f \in \RR^{(n-1) \times n}$ by
    \[ (f)_{i,j} = \begin{cases} 1 & \text{ if } i=1, \  \sigma(j) > \ell +1,\\
    -1 & \text{ if } i=2, \ \sigma(j)\geq \ell,\\
    0 & \text{ otherwise}\end{cases}
    \]
We note that $f \in w^\perp$ and so we may define the tropical map $\varphi^{\sigma, \ell} := \varphi_{w,F}$ where $F = \conv\{0, f\}$.
\end{definition}

\begin{figure}
    \centering
    \includegraphics[scale=0.35]{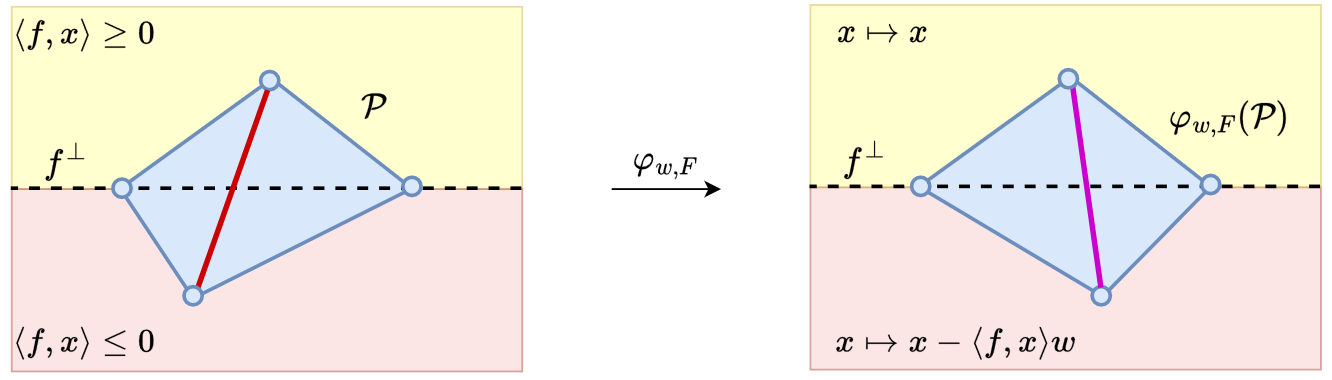}
    \caption{{A diagram of a combinatorial mutation. To check that the image of the tropical map $\varphi_{w, F}$ does not have any \textit{new vertices}, we check that every red line segment, as in the left-hand figure, is in the relative interior of a $2$-dimensional face of the polytope. To check that the image of $\varphi_{w,F}$ is convex, we show that every purple line segment, as in the right-hand figure, is in the interior of the image.}}
    \label{fig: mutation no new vertices}
\end{figure}

\begin{lemma} \label{lemma: flag inner product}
Fix $k<n$ and \Cref{setup flag}. For each $v \in V(\mathcal{P}_\sigma^k)$ we have $\langle f, v \rangle\in \{-1,0,1\}$. More precisely:
\begin{enumerate}[label=(\arabic*)]
    \item $\langle f, v \rangle=1$ if and only if $v=v_I^k$ with $I=\{j_1,j_2< j_3<\dots<j_k\}$ such that $(k=1$ and $\sigma(j_1)>\ell+1)$ or $(k\geq 2$ and $\sigma(j_1)>\ell+1>\ell>\sigma(j_2))$;
    \item $\langle f, v \rangle=-1$ if and only if $v=v_I^k$ with $k\geq 2$ and $I=\{\lambda,\mu< i_3<\dots<i_k\}$.
\end{enumerate}
\end{lemma}

\begin{proof}
Note that if $k=1$, then the entries of $f \in \RR^{1\times n} \subseteq \RR^{(n-1) \times n}$ are either $0$ or $1$, hence $\langle v_I,f \rangle \in \{0,1\}$. Assume now that $k \ge 2$.

Let $v = v_I \in V(\MP^k_\sigma)$ and suppose that $I=\{\lambda< \mu<i_3<\dots<i_k\}$. Note that the only non-zero coordinates of $f$ lie in the first two rows and $\ell=\sigma(\lambda) < \sigma(\mu) = \ell+1$. Therefore, we have
\[ \langle f, v \rangle =  (f)_{1,\mu} + (f)_{2,\lambda} = 0-1 = -1.\]
Suppose that $I=\{i_1<i_2<i_3<\dots<i_k\}$ with $\{i_1,i_2\} \neq \{ \lambda, \mu\}$. Let $a, b \in \{1, 2 \}$ be indices such that $\sigma(i_a) < \sigma(i_b)$. It follows that the inner product of $f$ and $v$ is given by
\[
\langle f,v\rangle = (f)_{1,i_b}+(f)_{2,i_a} = \begin{cases} 0+0=0 & \text{ if } \sigma(i_b)\leq \ell+1,\\
1+0 = 1 & \text{ if } \sigma(i_a)< \ell \text{ and } \sigma(i_b)>\ell+1,\\
1-1 = 0 & \text{ if } \sigma(i_a) \geq \ell,
\end{cases}\]
In particular, $\langle f, v_I \rangle\in \{-1,0,1\}$ and so (1) and (2) hold.
\end{proof}

\begin{prop}\label{prop: tropical map bijection flag}
    Fix $k<n$ and \Cref{setup flag}. We have $\varphi^{\sigma, \ell}(V(\MP_\sigma^k)) = V(\MP_\tau^k)$ is a bijection between the vertices of $\MP_\sigma^k$ and $\MP_\tau^k$ that preserves the indexing of the vertices.
\end{prop}

\begin{proof}
    Recall that the tropical map $\varphi^{\sigma, \ell} := \varphi_{w, F}$ acts as follows
    \[
    \varphi^{\sigma, \ell}(x) = 
    \begin{cases}
        x & \text{if } \langle f, x \rangle \ge 0, \\
        x - \langle f,x \rangle w & \text{if } \langle f,x\rangle < 0
    \end{cases}
    \]
    for each $x \in \RR^{(n-1) \times n}$. By \Cref{lemma: flag inner product}, the only vertices of $\MP_\sigma^k$ that are not fixed by $\varphi^{\sigma, \ell}$ are indexed by subsets $I = \{\lambda, \mu < i_3 < \dots < i_k \}$. For each such vertex $v_I$ of $\MP_\sigma^k$, we have that $v_I - w$ is the vertex of $\MP_\tau^k$ indexed by $I$. Since $\MP_\sigma^k$ and $\MP_\tau^k$ have the same number of vertices, it follows that $\varphi^{\sigma, \ell}$ acts as a bijection on the vertices.
\end{proof}

To describe the permutations that we consider, we make use of the language of permutation patterns.

\begin{definition}\label{def : permutation pattern}
A permutation $\sigma\in S_n$ written in single-line notation $(\sigma(1), \dots,  \sigma(n))$ is said to \textit{contain the pattern $\pi = (\pi_1, \dots, \pi_k) \in S_k$} for some $k\leq n$ if there is a subsequence $(\sigma(i_1), \sigma(i_2), \dots, \sigma(i_k))$ for some $1\leq i_1<i_2<\dots<i_k\leq n$, such that the relative order of the subsequence coincides with that of $\pi$. Explicitly, for all $p, q \in [k]$ we have that $\sigma_{i_p} < \sigma_{i_q}$ if and only if $\pi_p < \pi_q$. If $\sigma$ does not contain the pattern $\pi$, then we say that $\sigma$ is free of pattern $\pi$.
\end{definition} 

For example, the only permutation in $S_n$ that is free of the pattern $12$ is $w_0 = (n, n-1, \dots, 1)$. For permutations on a few indices, we omit the brackets from the single line notation. For example the identity permutation in $S_4$ is represented by $1234$. 

\smallskip

We proceed by showing the following technical lemmas (Lemmas~\ref{lemma: sum vertices flag}, \ref{lemma: bad patterns flag} and \ref{prop: mutation flag}) that we will use to prove \Cref{thm: DSR_2 Flag varieties}. The overarching idea of the proof is shown in \Cref{fig: mutation no new vertices}. Assume \Cref{setup flag}, let $\varphi_{w, F}$ be the tropical map as in \Cref{def flag tropical map} with $F = \conv\{0, f\}$ and let $\MP_{B^\sigma}$ be the matching field polytope associated to $B^\sigma$. Suppose that for every pair of vertices $v_1, v_2 \in V(\MP_{B^\sigma})$ with $\langle v_1, f\rangle = 1 = -\langle v_2, f\rangle$ there exist vertices $u_1, u_2 \in V({\mathcal{P}_{B^\tau}}) \cap f^\perp$ such that $v_1 + v_2 = u_1 + u_2$. Then it follows that $\varphi_{w, F}(\MP_{B^\sigma}) \subseteq \conv\{\varphi_{w, F}(v) \mid v \in V({\mathcal{P}_{B^\tau}}) \} = \MP_{B^\tau}$.

\begin{lemma} \label{lemma: sum vertices flag}
Assume that $\sigma$ is free of the patterns $1423$ and $4123$. Furthermore, assume that $\mu = \max \{ i \in [n] \mid \sigma(i) > n+1-i\}$. Then for every $v^k_I \in V(\mathcal{P}_\sigma^k)$ and $v^h_J \in V(\mathcal{P}_\sigma^h)$ with $\langle f, v_I^k \rangle=-1$, $\langle f, v_J^h \rangle=1$, there exist $u_1\in V(\mathcal{P}_\sigma^k) \cap f^\perp$, $u_2\in V(\mathcal{P}_\sigma^h) \cap f^\perp$ such that $v_I^{{k}}+v_J^{{h}} = u_1+u_2$.
\end{lemma}
\begin{proof}
Let $v^k_I \in V(\mathcal{P}_\sigma^k)$ and $v^h_J \in V(\mathcal{P}_\sigma^h)$ such that $\langle f, v_I^k \rangle=-1$, $\langle f, v_J^h \rangle=1$. Note that $k\geq 2$. Moreover by \Cref{lemma: flag inner product} it follows that $I=\{ \lambda, \mu< i_3<\dots< i_k \}$ and $J= \{ j_1,j_2 < j_3 <\dots < j_h \}$ such that $\sigma(j_1) > \ell+1$ and if $h \ge 2$ then we also have $\ell > \sigma(j_2)$. By the assumption on $\mu$, it follows that $j_1 < \mu$. Since $\sigma$ is free of $1423$ and $4123$, if $h \ge 2$ then it follows that $\max\{j_1, j_2\} > \lambda$. We proceed by taking cases on $h$.
\setcounter{case}{0}
\begin{case} Assume $h = 1$. Then $J=\{j_1\}$ with $\sigma(j_1)>\ell+1$. Since $j_1 < \mu < i_3$, we have
\[ v_I^k + v_J^h = \left( e_{1,\mu}+e_{2,\lambda}+ \sum_{a=3}^k e_{a,i_a} \right) + e_{1,j_1} = \left( e_{1,j_1}+e_{2,\lambda}+ \sum_{a=3}^k e_{a,i_a} \right) + e_{1,\mu} = u_1 + u_2\]
where
$u_1 = v^k_{\{ j_1,\lambda<i_3<\dots<i_k \}} \in V(\mathcal{P}_\sigma^k)$ and $u_2 = v^h_{\{\mu\}} \in V(\mathcal{P}_\sigma^h)$, which satisfy $\langle f, u_1\rangle=\langle f,u_2\rangle=0$.
\end{case}

\begin{case}
Assume $h=k=2$. Then $I=\{\lambda,\mu\}$ and $J=\{j_1,j_2\}$ with $\sigma(j_2) < \ell < \ell+1 < \sigma(j_1)$. We can write
\[
v_I^k + v_J^h = \left( e_{1,\mu}+e_{2,\lambda}\right) + \left(e_{1,j_1}+e_{2,j_2}\right)
= \left( e_{1,j_1}+e_{2,\lambda} \right) + \left(e_{1,\mu}+e_{2,j_2}\right) = u_1 + u_2
\]
with $u_1=e_{1,j_1}+e_{2,\lambda}$ and $u_2=e_{1,\mu}+e_{2,j_2}$. 
\end{case}

\begin{case}
Assume $h=2$ and $k \geq 3$. By assumption, we have that $J=\{j_1,j_2\}$ with $\sigma(j_2) < \ell < \ell+1 < \sigma(j_1)$ and $j_1 < \mu< i_3$, hence
\[
v_I^k + v_J^h = \left( e_{1,\mu}+e_{2,\lambda}+ \sum_{a=3}^k e_{a,i_a} \right) + \left(e_{1,j_1}+e_{2,j_2}\right)
= \left( e_{1,j_1}+e_{2,\lambda}+ \sum_{a=3}^k e_{a,i_a} \right) + \left(e_{1,\mu}+e_{2,j_2}\right) = u_1 + u_2
\]
where $u_1 = v^k_{\{ j_1,\lambda<i_3<\dots<i_k \}} \in V(\mathcal{P}_\sigma^k)$ and $u_2 = v^h_{\{\mu,j_2\}} \in V(\mathcal{P}_\sigma^h)$, which satisfy $\langle f, u_1\rangle=\langle f,u_2\rangle=0$.
\end{case}

\begin{case}
Assume $h \geq 3$. Then $J= \{j_1,j_2<j_3<\dots<j_h\}$ with $\sigma(j_2)< \ell <\ell+1<\sigma(j_1)$. If $k = 2$ then note that $\max\{j_1, j_2 \} > \lambda$, and so we have
\[
v_I^k +v_J^h = \left( e_{1,\mu}+ e_{2,\lambda} \right) + \left( \sum_{b=1}^h e_{b,j_b} \right)
= \left( e_{1,\mu}+e_{2,j_2}\right) + \left( e_{1,j_1}+e_{2,\lambda} + \sum_{b=3}^h e_{b,j_b} \right) = u_1 + u_2
\]
where $u_1 = v^k_{\mu,j_2} \in V(\mathcal{P}_\sigma^k)$ and $u_2 = v^h_{j_1,\lambda<j_3<\dots<j_h} \in V(\mathcal{P}_\sigma^h)$, which satisfy $\langle f, u_1\rangle=\langle f,u_2\rangle = 0$. Otherwise, if $k \ge 3$ then we take further cases with either $j_2 < \mu$ or $j_2 > \mu$.
\begin{subcase}
Assume that $j_2 < \mu$. In this case we have that
\[
v_I^k +v_J^h 
= \left( e_{1,\mu}+ e_{2,\lambda} + \sum_{a=1}^k e_{a,i_a} \right) + \left( \sum_{b=1}^h e_{b,j_b} \right)
= \left( e_{1,\mu}+e_{2,j_2}+\sum_{a=3}^k e_{a,i_a} \right) + \left( e_{1,j_1}+e_{2,\lambda} + \sum_{b=3}^h e_{b,j_b} \right) 
= u_1 + u_2
\]
where $u_1 = v^k_{\mu,j_2<i_3<\dots<i_k} \in V(\mathcal{P}_\sigma^k)$ and $u_2 = v^h_{j_1,\lambda<j_3<\dots<j_h} \in V(\mathcal{P}_\sigma^h)$, which satisfy $\langle f, u_1\rangle=\langle f,u_2\rangle = 0$.
\end{subcase}

\begin{subcase}
Assume that $j_2 > \mu$. In this case we have that
\[
v_I^k +v_J^h 
= \left( e_{1,\mu}+ e_{2,\lambda} + \sum_{a=1}^k e_{a,i_a} \right) + \left( \sum_{b=1}^h e_{b,j_b} \right)
= \left( e_{1,j_1} + e_{2,\lambda} + \sum_{a=3}^k e_{a,i_a} \right) + \left( e_{1,\mu}+ e_{2,j_2} + \sum_{b=3}^h e_{b,j_b} \right) 
= u_1 + u_2
\]
where $u_1 = v^k_{j_1,\lambda<i_3<\dots <i_k} \in V(\mathcal{P}_\sigma^k)$ and $u_2 = v^h_{\mu,j_2<j_3<\dots< j_h} \in V(\mathcal{P}_\sigma^h)$, which satisfy $\langle f, u_1\rangle=\langle f,u_2\rangle = 0$.
\end{subcase}
\end{case}

In each case we have constructed $u_1 \in V(\mathcal{P}_\sigma^k)$ and $u_2 \in V(\mathcal{P}_\sigma^h)$ such that $v_I^k + v_J^h = u_1 + u_2$ and $\langle f, u_1\rangle=\langle f,u_2\rangle = 0$, so we are done.
\end{proof}

An easy consequence of the lemma above is the following result which shows that every point in the Minkowski sum $\mathcal{P}_\sigma^{{K}}$ can be written as a sum of vertices that all lie in the same half-space defined by $f$.

\begin{lemma} \label{lemma: bad patterns flag}
Assume \Cref{setup flag} holds and that $\mu = \max\{ i \in [n] \mid \sigma(i) > n+1-i\}$. If $\sigma$ is free of patterns $1423, 4123, 1324$ and $3124$, then so is $\tau = (\ell \ \ell+1) \sigma$.
\end{lemma}

\begin{proof}
Assume by contradiction that $\tau$ contains the pattern $1423$. Then it follows that $\ell$ and $\ell+1$ correspond to elements of the pattern, otherwise the same pattern exists in $\sigma$. Since $\lambda < \mu$ by assumption, therefore $\ell$ and $\ell+1$ correspond to $3$ and $4$ in the pattern respectively. However $\sigma$ contains the pattern $(\ell \ \ell+1) 1423 = 1324$, a contradiction. The same argument shows that $\tau$ is also free of the pattern $4123$.

Assume by contradiction that $\tau$ contains the pattern $3124$. As before, if $\ell$ and $\ell+1$ are not contained in the pattern then the same pattern appears in $\sigma$. Therefore $\ell$ and $\ell+1$ correspond to {$2$ and $3$} in the pattern respectively. By assumption for all $i > \mu$ we have $\sigma(i) < \ell+1$. However the $4$ in the pattern $3124$ contradicts this assumption since $\ell+1$ corresponds to $3$. Therefore $\tau$ is free of the pattern $3124$. A similar argument shows that $\tau$ is free of the pattern $1324$.
\end{proof}

\begin{lemma} \label{prop: mutation flag}
Assume \Cref{setup flag} holds. Suppose $\sigma$ is free of patterns $1423,4123, 1324$ and $3124$ and assume that $\mu = \max\{ i \in [n] \mid \sigma(i) > n+1-i\}$. Then $\varphi^{\sigma,\ell}$ defines a combinatorial mutation from $\mathcal{P}_\sigma^{{K}}$ to $\mathcal{P}_\tau^{{K}}$.
\end{lemma}

\begin{proof}
Write $\mathcal{P}_\sigma^{{K},+} = \mathcal{P}_\sigma^{{K}} \cap \{x \in \RR^{(n-1)\times n} \mid \langle f, x\rangle \geq 0\}$ and similarly $\mathcal{P}_\sigma^{{K},-} = \mathcal{P}_\sigma^{{K}} \cap \{x \in \RR^{(n-1)\times n} \mid \langle f,x\rangle\leq 0\}$. 
Note that $\sum_{k\in K} \mathcal{P}^{k,+}_\sigma \subseteq \mathcal{P}^{K,+}_\sigma$. Consider $t \in \mathcal{P}^{K,+}_\sigma$. We can write $t$ in the form 
\[ t = \sum_{k\in K}\sum_{I \in \binom{[n]}{k}} \alpha_I^k v_I^k\]
for some $\alpha_I^k\geq 0$ such that $\sum_I\in \binom{[n]}{k} \alpha_I^k =1$ for all $k \in K$. Suppose that there exist $k,h \in K$ and $I\in\binom{[n]}{k}$, $J\in\binom{[n]}{h}$ such that $\alpha_I^k\neq 0$, $\alpha_J^h\neq 0$ and $v_I^k, v_J^h$ lie in two different half-spaces defined by $f$, that is, $\langle f, v_I^k \rangle = -1$ and $\langle f, v_J^h\rangle =1$. By \Cref{lemma: sum vertices flag}, there exists $u_1\in V(\mathcal{P}^k_\sigma)\cap f^\perp$ and $u_2 \in V(\mathcal{P}^h_\sigma)\cap f^\perp$ such that $v^k_I + v^h_J = u_1 + u_2$. Then, if $\alpha = \min\{ \alpha^k_I, \alpha^h_J \}$, we can substitute $\alpha^k_I v^k_I + \alpha^h_J v^h_J$ with $\alpha(u_1+u_2) + (\alpha^k_I-\alpha) v^k_I + (\alpha^h_J-\alpha) v^h_J$. The new expression for $t$ contains a strictly smaller number of pair of vertices lying in opposite half-space defined by $f$. Since the sum is finite we can repeat the process until we get an expression such that for all $k\in K$ and $I\in \binom{[n]}{k}$ with $\alpha^k_I\neq 0$, $\langle f, v^k_I\rangle \geq 0$. Then $t \in \sum_{k\in K} \mathcal{P}^{k,+}$ and we can conclude that $\sum_{k\in K} \mathcal{P}^{k,+}_\sigma = \mathcal{P}^{K,+}_\sigma$.

Since the map $\varphi^{\sigma,\ell}$ is linear on $\mathcal{P}_\sigma^{{K},+}$ and on $\mathcal{P}_\sigma^{{K},-}$, we have
\[ \varphi^{\sigma,\ell}\left(\mathcal{P}_\sigma^{{K}}\right) = \varphi^{\sigma,\ell}\left(\mathcal{P}_\sigma^{{K},+}\right) \cup \varphi^{\sigma,\ell}\left(\mathcal{P}_\sigma^{{K},-}\right) = \sum_{k\in {K}} \varphi^{\sigma,\ell}\left(\mathcal{P}_\sigma^{k,+}\right) \cup \sum_{k \in {K}} \varphi^{\sigma,\ell}\left(\mathcal{P}_\sigma^{k,-}\right). \]

By \Cref{prop: tropical map bijection flag}, it follows that
\[ \sum_{k\in {K}} \varphi^{\sigma,\ell} \left(\mathcal{P}_\sigma^{k,+} \right) \cup \sum_{k \in {K}} \varphi^{\sigma,\ell} \left( \mathcal{P}_\sigma^{k,-} \right) = \sum_{k \in {K}} \mathcal{P}_\tau^{k,+} \cup \sum_{k \in {K}} \mathcal{P}_\tau^{k,-}. \]
To conclude, it suffices to prove that $\mathcal{P}_\tau^{{K},+} = \sum_{k \in {K}} \mathcal{P}_\tau^{k,+}$ and $\mathcal{P}_\tau^{{K},-} = \sum_{k \in {K}} \mathcal{P}_\tau^{k,-}$. To do this, it suffices to show that an analogous statement to \Cref{lemma: sum vertices flag} holds for $\tau$. The claim will follow analogously to the case for $\sigma$ at the beginning of this proof.
Fix $v_I^k \in V(\mathcal{P}^k_\tau)$ and $v^h_J \in V(\mathcal{P}^h_\tau)$ with $\langle f, v^k_I\rangle=-1$, $\langle f, v^h_J\rangle=1$, and $k \neq h$. We will construct $u_1 \in V(\mathcal{P}^k_\tau)\cap f^\perp$ and $u_2 \in V(\mathcal{P}^h_\tau)\cap f^\perp$ such that $v^k_I+v^h_J = u_1+u_2$. By choice of $v^k_I$ and $v^h_J$ and by \Cref{lemma: flag inner product}, it follows immediately that $I=\{\lambda,\mu<i_3<\dots<i_k\}$ and $J=\{j_1,j_2<j_3<\dots<j_h\}$ with $\sigma(j_1) > \ell+1$ and if $h \ge 2$ then we also have $\ell > \sigma(j_2)$. By the assumption on $\mu$, it follows that $j_1 < \mu$. Since $\tau$ is free of $1423$ and $4123$, if $h \ge 2$ then it follows that $\max\{j_1, j_2\} > \lambda$. We proceed by taking cases on $h$.

\setcounter{case}{0}
\begin{case} Assume $h = 1$. Since $j_1 < \mu$ we have
\[ v_I^k + v_J^h = \left( e_{1,\lambda}+e_{2,\mu}+ \sum_{a=3}^k e_{a,i_a} \right) + e_{1,j} 
= \left( e_{1,j_1}+e_{2,\mu}+ \sum_{a=3}^k e_{a,i_a} \right) + e_{1,\lambda} = v^k_{\{ j_1,\mu<i_3<\dots<i_k \}} + v^h_{\{\lambda\}}. \]
\end{case}

\begin{case}
Assume $h=2$. It follows that $k \geq 3$. Since $j_1 < \mu$ we have that
\begin{align*} v_I^k + v_J^h 
= \left( e_{1,\lambda}+e_{2,\mu}+ \sum_{a=3}^k e_{a,i_a} \right) + \left(e_{1,j_1}+e_{2,j_2}\right)
& = \left( e_{1,j_1}+e_{2,\mu}+ \sum_{a=3}^k e_{a,i_a} \right) + \left(e_{1,\lambda}+e_{2,j_2}\right) = u_1 + u_2 
\end{align*}
where $u_1 = v^k_{\{j_1,\mu<i_3<\dots < i_k\}}$ and $u_2 =  v^h_{\{\lambda,j_2\}}$.
\end{case}

\begin{case}
Assume $h \geq 3$. If $k = 2$ then we have that $\max\{j_1, j_2 \} \ge \lambda$ and so we have
\[
v_I^k +v_J^h 
= \left( e_{1,\lambda}+ e_{2,\mu}\right) + \left( \sum_{b=1}^h e_{b,j_b} \right)
= \left( e_{1,\lambda}+e_{2,j_2}\right) + \left( e_{1,j_1}+e_{2,\mu} + \sum_{b=3}^h e_{b,j_b} \right) 
= u_1 + u_2
\]
where $u_1 = v^k_{\lambda,j_2}$ and $u_2 = v^h_{j_1,\mu<j_3<\dots<j_h}$. Otherwise if $k \ge 3$ we take cases on whether $j_2 < \mu$ or $j_2 > \mu$.

\begin{subcase}
Assume that $j_2 < \mu$. Since $\max\{j_1, j_2 \} \ge \lambda$ we have
\[
v_I^k +v_J^h 
= \left( e_{1,\lambda}+ e_{2,\mu} + \sum_{a=1}^k e_{a,i_a} \right) + \left( \sum_{b=1}^h e_{b,j_b} \right)
= \left( e_{1,\lambda}+e_{2,j_2}+\sum_{a=3}^k e_{a,i_a} \right) + \left( e_{1,j_1}+e_{2,\mu} + \sum_{b=3}^h e_{b,j_b} \right) 
= u_1 + u_2
\]
where $u_1 = v^k_{\lambda,j_2<i_3<\dots<i_k}$ and $u_2 = v^h_{j_1,\mu<j_3<\dots<j_h}$.
\end{subcase}

\begin{subcase}
Assume that $j_2 > \mu$. Since $j_1 < \mu$ we have
\[
v_I^k +v_J^h 
= \left( e_{1,\lambda}+ e_{2,\mu} + \sum_{a=1}^k e_{a,i_a} \right) + \left( \sum_{b=1}^h e_{b,j_b} \right) 
= \left( e_{1,j_1} + e_{2,\mu} + \sum_{a=3}^k e_{a,i_a} \right) + \left( e_{1,\lambda}+ e_{2,j_2} + \sum_{b=3}^h e_{b,j_b} \right) 
= u_1 + u_2
\]
where $u_1 = v^k_{j_1,\mu<i_3<\dots <i_k}$ and $u_2 = v^h_{\lambda,j_2<j_3<\dots< j_h}$.
\end{subcase}
\end{case}
So in each case we have constructed vertices $u_1 \in V(\mathcal{P}_\tau^k)$ and $u_2 \in V(\mathcal{P}_\tau^h)$ such that $v_I^k + v_J^h = u_1 + u_2$ and $\langle f, u_1\rangle=\langle f,u_2\rangle = 0$.
Hence we can conclude that
\[ \varphi^{\sigma,\ell}\left(\mathcal{P}_\sigma^{{K}}\right) = \sum_{k \in {K}} \mathcal{P}_\tau^{k,+} \cup \sum_{k \in {K}} \mathcal{P}_\tau^{k,-} = \mathcal{P}_\tau^{{K},+} \cup \mathcal{P}_\tau^{{K},-} = \mathcal{P}_\tau^{{K}}. \tag*{\qedhere}\] 
\end{proof}

\begin{proof}[\bf Proof of \Cref{thm: DSR_2 Flag varieties}]
Let $\sigma \in S_n$ be a permutation free of the patterns $1423,4123,3124$ and $1324$. We prove that there is a sequence of combinatorial mutations from $\mathcal{P}_\sigma^{{K}}$ to $\mathcal{P}^{{K}}_{w_0}$ by reverse induction on the inversion number of $\sigma$. 

Since $w_0$ has the maximum inversion number, the base case is trivial. Assume now $\sigma \neq w_0$. Let $\mu = \max \{i \in [n] \mid \sigma(i) > n+1-i\}$. Define $\ell = \sigma^{-1}(\mu) - 1$ and $\lambda = \sigma(\ell)$. By the definition of $\mu$, it immediately follows that $\lambda < \mu$ and so $\sigma$ and $\ell$ satisfy the conditions of \Cref{setup flag}. By \Cref{lemma: bad patterns flag}, we have that $\tau = (\ell \ \ell+1) \sigma$ is free of patterns $1423, 4123, 1324$ and $3124$.
Since the inversion number of $\tau$ is one more than the inversion number of $\sigma$, by induction %we have that 
there is a sequence of combinatorial mutations from $\mathcal{P}_\tau^{{K}}$ to $\mathcal{P}_{w_0}^{{K}}$. By \Cref{prop: mutation flag} there is a combinatorial mutation from $\mathcal{P}_\sigma^{{K}}$ to $\mathcal{P}_\tau^{{K}}$, which completes the proof. 
\end{proof}

\begin{corollary}
The $2$-block diagonal matching field polytopes are combinatorial mutation equivalent and give rise to toric degenerations.
\end{corollary}

\begin{proof}
Consider a permutation $\sigma = (\ell, \ell-1, \dots, 1, n, n-1, \dots, 1)$ for some $\ell \in \{0,\dots,n-1\}$. Observe that the only patterns of length four which appear in $\sigma$ are $4321$, $3214$, $2143$ and $1432$. In particular, $\sigma$ avoids the patterns $4123$, $3124$, $1423$ and $1324$. Hence, by \Cref{thm: DSR_2 Flag varieties}, we have that ${\mathcal{P}_{\sigma}}$ is combinatorial mutation equivalent to the diagonal matching field polytope $\mathcal{P}_{w_0}$.
\end{proof}

\begin{remark}
    We expect that all the matching fields giving rise to toric degenerations of partial flag varieties are related by a sequence of mutations. However, the challenge in proving this via an explicit sequence of polytopes is that the intermediate polytopes may not necessarily be lattice polytopes. See \Cref{example: gr36 mutation with non lattice polytope}. However, we note that, in this example, the intermediate polytopes are still examples of NO-bodies for the  Grassmannian. In the following section, we computationally examine the matching field polytopes that give rise to toric degenerations.
\end{remark}
\begin{example}\label{example: gr36 mutation with non lattice polytope}
Consider the matching field polytope $\mathcal{P}_\sigma$ given by $\sigma = (6 \, 2 \, 4 \, 3 \, 5 \, 1)$ for $\Gr(3,6)$. Note that $\sigma$ contains the forbidden pattern $1324$ on the indices $2,3,4,5$. We want to construct a sequence of combinatorial mutations from the polytope $\mathcal{P}_\sigma$ to $\mathcal{P}_\tau$ where $\tau = ( 6 \, 2 \, 5 \, 3 \, 4 \, 1)$.
Let
\begin{align*}
    f_1 = \begin{pmatrix}
    0 & 0 & 0 & -1 & 0 & 0\\
    0 & 1 & 0 & 0 & 0 & 0\\
    0 & 0 & -1 & -1 & 0 & 0 \end{pmatrix}, \quad
    & w_1 = \begin{pmatrix}
    0 & 0 & 0 & 0 & 0 & 0\\
    0 & 0 & 1 & -1 & 0 & 0\\
    0 & 0 & -1 & 1 & 0 & 0 \end{pmatrix},\\
    f_2 = \begin{pmatrix}
    1 & 0 & 0 & 0 & 0 & 0\\
    -1 & 0 & -1 & 0 & -1 & 0\\
    0 & 0 & -1 & 0 & 0 & 0 \end{pmatrix}, \quad
    & w_2 = \begin{pmatrix}
    0 & 0 & 1 & -1 & 0 & 0\\
    0 & 0 & -1 & 1 & 0 & 0\\
    0 & 0 & 0 & 0 & 0 & 0 \end{pmatrix}.
\end{align*}
and let $F_i = \conv(0,f_i)$ for $i=1,2$.
It is possible to show that $\mathcal{P}_\tau = \varphi_{-w_1, F_1} \circ \varphi_{w_2,F_2} \circ \varphi_{w_1,F_1} (\mathcal{P}_\sigma)$ and this is a sequence of combinatorial mutations. Note that the resulted polytope $Q$ after applying the two mutations, $\varphi_{w_2,F_2} \circ \varphi_{w_1,F_1} (\mathcal{P}_\sigma)$, is not a lattice polytope, since it has one rational vertex:
\[ \begin{pmatrix}
1/2 & 0 & 0 & 0 & 1/2 & 0\\
0 & 1/2 & 1/2 & 0 & 0 & 0\\
0 & 0 & 0 & 1/2 & 0 & 1/2 \end{pmatrix} = \frac 12(v_{134} + v_{526}),
\]
where $v_{134}$ and $v_{526}$ are vertices of $\MP_\tau$. Our computation shows that the polytope $\mathcal{Q}$ can be obtained from the polytope of the hexagonal matching field in Figure~\ref{fig:tropical line arrangements}, by adding the extra rational point above. 

Note that \Cref{thm: DSR_2 Flag varieties} does not apply to the matching field $B^\tau$ as it contains the pattern $1423$ on the indices $2,3,4,5$. Using the \textit{MatchingFields} package \cite{ClarkeMatchingFieldsPackage} for \textit{Macaulay2} \cite{M2}, it is straightforward to check that $B^\sigma$ gives rise to a toric degeneration of $\Gr(3,6)$. By the sequence of mutations above, we have that $B^\tau$ also gives rise to a toric degeneration of $\Gr(3,6)$.
\end{example}

\section{Computational results}\label{sec: computational results}

We have seen in \Cref{example: gr36 mutation with non lattice polytope} that there exist matching fields $B^\sigma$ that give rise to toric degenerations of partial flag varieties that are not covered by \Cref{thm: DSR_2 Flag varieties}. However, in such cases, the matching field polytopes may be mutable. That is, the polytopes can be transformed into one another by a sequence of combinatorial mutations.

In this section, we consider a class of matching fields that generalize those in \Cref{def:matching field B sigma}. Moreover, we provide a complete computational classification of all toric degenerations obtained from these matching fields for small Grassmannians and flag varieties, namely $\Gr(3,6)$, $\Gr(3,7)$, $\Fl_4$, and $\Fl_5$.  In Examples~\ref{example: gr36 all cones}, \ref{example: gr37 all cones}, \ref{example: compute fl4} and \ref{example: compute fl5}, the matching field polytopes are constructed using the \textit{MatchingFields} package \cite{ClarkeMatchingFieldsPackage} for \textit{Macaulay2} \cite{M2}, and their properties and isomorphism classes are computed using \textit{Polymake} \cite{polymake:2000}. %

\begin{definition}\label{def: scale the second row weight matrix}
Fix $n \ge 2$ and let $p \ge n+1$ be a prime number. For each natural number $c \ge 1$ coprime to $p$ and permutation $\sigma \in S_n$ we define the weight matrix
\[
M_c^\sigma = 
\begin{bmatrix}
0           & 0             & \dots & 0             & 0\\
c\sigma(1)  & c\sigma(2)    & \dots & c\sigma(n-1)  & c\sigma(n) \\ 
np          & (n-1)p        & \dots & 2p            & p   \\
np^2        & (n-1)p^2      & \dots & 2p^2          & p^2 \\
\vdots      & \vdots        & \ddots& \vdots        & \vdots \\
np^{n-2}    & (n-1)p^{n-2}  & \dots & 2p^{n-2}      & p^{n-2} \\
\end{bmatrix}.
\]
We now show that each $M_c^\sigma$ induces a coherent matching field which we denote by $B_c^\sigma$.
\end{definition}

\begin{prop}
For all $n \ge 2$, $c \in [p-1]$ and permutations $\sigma \in S_n$, the weight matrix $M_c^\sigma$ induces a coherent matching field for the flag variety.
\end{prop}

\begin{proof}
Fix a subset $I = \{j_1, \dots, j_k \} \subseteq [n]$. 
For each $\tau \in S_k$, we define
$w_\tau := \sum_{i = 1}^{k} M_{i, j_{\tau(i)}}.$
We proceed by showing that the values $w_\tau$ are distinct for each $\tau \in S_k$. Let $\tau, \tau' \in S_k$ be permutations and assume that $w_\tau = w_{\tau'}$. By definition of the weight matrix $M_c^\sigma$, we may write 
\[
w_\tau      = \sum_{i = 1}^{k-2} a_i p^i + ca_0 
\quad \text{and} \quad
w_{\tau'} = \sum_{i = 1}^{k-2} b_i p^i + cb_0 \]
with $a_i, b_i \in [n]$ for all $i \in \{0, 1, \dots, k-2 \}$. Note that $p$ is a prime number and $c, a_0, b_0$ are all coprime to $p$. Since we have that $0 = w_\tau - w_{\tau'} = c(a_0 - b_0) \mod p$, it follows that $a_0 = b_0$.

The next step is to show that $a_i = b_i$ for all $i \in \{1, \dots, k-2 \}$. Assume by contradiction that $d \in \{1, \dots, k-2\}$ is the largest index such that $a_d \neq b_d$. Without loss of generality we may assume that $a_d > b_d$. For each $i \in \{1, \dots, d\}$ we have that $a_i - b_i \ge -n \ge 1-p$. It follows that
\[
w_\tau - w_{\tau'} = (a_d - b_d)p^d + \sum_{i = 1}^{d-1} (a_i - b_i)p^i \ge p^d - \sum_{i = 1}^{d-1} (p-1)p^i = p > 0.
\]
By assumption we have $w_\tau = w_{\tau'}$ which is a contradiction. So we have shown that for any pair of permutations $\tau, \tau' \in S_k$, if $w_\tau = w_{\tau'}$ then $\tau = \tau'$. Hence the values $w_\tau$ are distinct for all $\tau \in S_k$. In particular, the minimum value $w_\tau$ is attained by a unique permutation $\tau \in S_k$. 
So, following \Cref{def coherent matching field}, the expression $\argmin_\tau(w_\tau)$ gives a well-defined permutation for any subset $I \subseteq [n]$, and so the weight matrix $M_c^\sigma$ induces a matching field.
\end{proof}

We use \textit{Macaulay2} \cite{ClarkeMatchingFieldsPackage, M2} to confirm that certain matching fields $\Bsig_c$ give rise to toric degenerations of $\Gr(k,n)$ and $\Fl_n$. To do this, we compute the initial ideal of the corresponding Pl\"ucker ideal with respect to the weight vector $w$ induced by the matrix $M^\sigma_c$. Whenever the matching field gives rise to a toric degeneration, it follows that $w$ lies in the corresponding tropicalization. In the following examples, we describe the maximal cones of the tropical varieties that contain these weight vectors $w$. We do this by comparing the polytopes. Note that the maximal cones of the tropicalization are listed up to symmetry under the symmetric group $S_n$. Since the $f$-vector of a matching field polytope is invariant under $S_n$, we compute the $f$-vectors of the matching field polytopes, using \textit{Polymake} \cite{polymake:2000}, and compare them to the already-known $f$-vectors associated to the maximal cones of the tropicalizations of $\Gr(k,n)^{\circ}$ and $\Fl_n^{\circ}$. We note that some of the polytopes associated to non-isomorphic cones may have the same $f$-vector. In this case, we differentiate the polytopes by comparing their face lattices.

\begin{example}[The Grassmannian $\Gr(3,6)$]\label{example: gr36 all cones}
\begin{table}[ht]
        \centering
            \begingroup\fontsize{8}{1em}\selectfont
\begin{tabular}{cccc}
            \toprule
            %\cline{2-3}
            Cone type & $c$ & $\sigma$ & Matching field polytope f-vector \\
            \midrule
            EEEG    & $3$ & $(6,4,3,5,2,1)$ & $ (20, 123, 386, 728, 882, 700, 358, 111, 18)$ \\
            EEFF(a) & $1$ & $(6,5,4,3,2,1)$ & $ (20, 122, 372, 670, 766, 571, 276, 83, 14)$ \\
            EEFF(b) & $1$ & $(1,6,5,4,3,2)$ & $ (20, 122, 376, 690, 807, 615, 302, 91, 15)$ \\
            EEFG    & $1$ & $(4,3,2,1,6,5)$ & $ (20, 122, 378, 701, 832, 645, 322, 98, 16)$ \\
            EFFG    & $1$ & $(2,1,6,5,4,3)$ & $(20, 122, 376, 690, 807, 615, 302, 91, 15)$ \\
            \midrule
            FFFGG &  \multicolumn{3}{l}{a non-prime cone from \Cref{example: gr36 hexagon}} \\
            EEEE  &  \multicolumn{3}{l}{a prime cone that cannot be obtained from a matching field} \\
            \bottomrule
        \end{tabular}
        \caption{Maximal cones of $\Trop(\Gr(3,6)^{\circ})$ and their associated matching field polytopes $B_c^\sigma$.}
        \label{table: gr36 polytopes}
        \endgroup
\end{table}There are seven top-dimensional cones of $\Trop(\Gr(3,6)^{\circ})$ up to isomorphism~\cite{speyer2004tropical}. Of these cones, six are prime. 
Of these six prime cones, five of them may be realised by matching fields $B_c^\sigma$. In \Cref{table: gr36 polytopes} below we list the type of the maximal cones of $\Trop(\Gr(3,6)^{\circ})$ together with an example of a matching field that gives rise to the same toric degeneration.

The isomorphism type of the toric variety of a matching field $\Lambda$ is characterised in \cite{Mohammadi2019} by its corresponding \textit{tropical line arrangement $T(\Lambda)$}, namely by the shape of the $(2,2,2)$-cell of $T(\Lambda)$.
A straightforward computation shows that for $c = 1$, the $(2,2,2)$-cell of $T(B_c^\sigma)$ cannot be a triangle for any permutation $\sigma \in S_6$.     However, this is possible for $c = 3$ and $\sigma = 643521$. See \Cref{fig:tropical line arrangements} (left-most). 
We note that the cones of type EEEE cannot be realised by any matching field. See \cite[Table~1]{Mohammadi2019}.

    \begin{figure}
        \centering
        \includegraphics[scale=0.7]{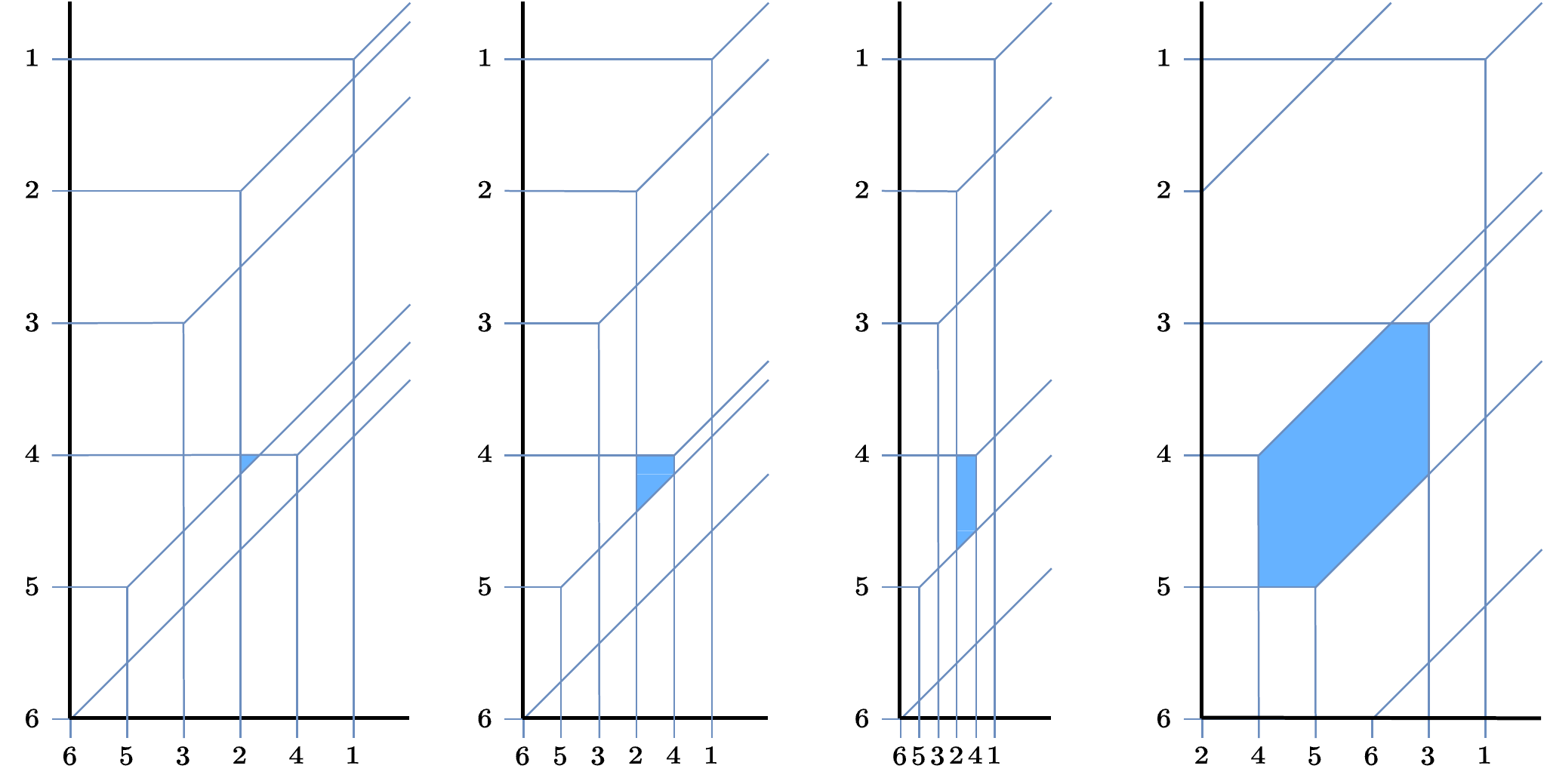}
        \caption{Tropical line arrangements associated to $B_c^\sigma$ with the $(2,2,2)$-cell shaded. The right-most line arrangement corresponds to $B_{3}^{615234}$. From left to right the other line arrangements correspond to $\sigma = 643521$ and $c = 1$, $2$ and $3$ respectively.}
        \label{fig:tropical line arrangements}
    \end{figure}
 \end{example}

In this paper, we have shown that, for all Grassmannians, the matching field polytopes of $B_c^\sigma$ are mutation equivalent when $c = 1$ and $\sigma$ is free a certain collection of permutations. We believe that our proofs can be extended to all matching fields $B_c^\sigma$, for any $c$.
Note that for all $c \in [p-1]$, the matching fields $B_c^{w_0}$ where $w_0 = (n, n-1, \dots, 1)$ give rise to the Gelfand-Tsetlin polytope.

\begin{conjecture}
    The matching field polytopes associated to $B_c^\sigma$ that give rise to toric degenerations of $\Gr(k,n)$ are all mutation equivalent for $\sigma \in S_n$ and $c \in [p-1]$.
\end{conjecture}

\begin{example} \label{example: gr36 hexagon}
    In the case $c = 3$, it is shown in \cite{Mohammadi2019} that a matching field for $\Gr(3,6)$ does not give rise to a toric degeneration if and only if the bounded $(2,2,2)$-cell of its tropical line arrangement is hexagonal. We note that if $c = 3$, then hexagonal matching fields appear among matching fields of the form $B_c^\sigma$.
    For example if $\sigma = 615234$ then the weight matrix $M_c^\sigma$ 
      gives rise to the tropical line arrangement in \Cref{fig:tropical line arrangements} (right-most). We note that the polytope $\mathcal{P}$ associated to this matching field is not mutation equivalent to the Gelfand-Tsetlin polytope. One way to see this is by noting that the normalised volume of $\mathcal{P}$ is $38$ whereas the normalised volume of the Gelfand-Tsetlin polytope is $42$.
\end{example}

\begin{example}[The Grassmannian $\Gr(3,7)$]\label{example: gr37 all cones}
The tropicalization of $\Gr(3,7)^{\circ}$ has a total of $125$ maximal cones, which were computed in \cite{herrmann2009draw}. Using \textit{Macaulay2}, we have observed that $69$ of these cones are prime.
However, $5$ of them do not arise from matching fields, as they contain a  
\textit{sub-weight} of type EEEE. That is, for each such weight vector $w$ there exists an index $i \in [7]$ such that the restriction of $w$ to the coordinates $J \subseteq [6]$ with $i \notin J$, lies in a cone of type EEEE in the tropicalization of $\Gr(3,6)^{\circ}$.
Using { weight matrices} $M_c^\sigma$ with $c \le 6$, we can obtain $40$ of the prime cones. Following the numbering conventions in \cite{herrmann2009draw} listed on the corresponding \href{https://users-math.au.dk/jensen/Research/G3_7/grassmann3_7.html}{{website}}, we list the prime cones of $\Trop(\Gr(3,7)^{\circ})$ together with their associated matching fields in Table~\ref{tab: gr37 matching field polytopes}.
For each prime cone, we compute the toric polytope associated to the corresponding initial ideal.
We identify whether that polytope is isomorphic to any matching field polytope $\mathcal{P}_{B_c^\sigma}$ where $c \in [6]$ and $\sigma \in S_7$. For example, the polytope associated to Gr\"obner cone number $14$ is isomorphic to the matching field polytopes for $B_2^{1463725}$, $B_3^{1453627}$, $B_5^{1365724}$ and $B_6^{1365724}$ and is not isomorphic to any matching field polytope $B_1^\sigma$.

{
    \centering
    \begingroup\fontsize{5.9}{1em}\selectfont
    \begin{longtable}[ht]{lll}
        \toprule
        Gr\"obner
        cone index & Polytope f-vector & List of $(c, \sigma)$ such that $B_c^\sigma$ gives the same toric degeneration \\
        \midrule 
        \endhead
$4$ & \texttt{35 335 1635 4918 9927 14022 14106 10101 5053 1691 347 36} & None -- has EEEE as a sub-weight \\ %does not arise from a matching field %None \\ % sub-weight of type EEEE remove 2, apply (1,3,4,5,6,7) -> (1,3,5,6,2,4)
$6$ & \texttt{35 335 1631 4885 9807 13770 13770 9807 4885 1631 335 35} & None -- has EEEE as a sub-weight \\ %does not arise from a matching field \\ % sub-weight of type EEEE remove 3, apply (1,2,4,5,6,7) -> (1,2,6,5,4,3)
$8$ & \texttt{35 334 1624 4866 9787 13784 13840 9905 4961 1666 344 36} & None \\ % Does not contain a sub-weight of type EEEE
$9$ & \texttt{35 333 1613 4813 9639 13518 13518 9639 4813 1613 333 35} & None \\ % Does not contain a sub-weight of type EEEE
$10$ & \texttt{35 333 1618 4854 9787 13826 13924 9989 5009 1681 346 36} &  None -- has EEEE as a sub-weight \\ %does not arise from a matching field \\ % sub-weight of type EEEE remove 1, apply (2,3,4,5,6,7) -> (1,2,6,5,3,4)
\hline
$13$ & \texttt{35 333 1601 4709 9240 12629 12251 8442 4064 1314 264 28} &  None -- has EEEE as a sub-weight \\ %does not arise from a matching field \\ % sub-weight of type EEEE remove 1, apply (2,3,4,5,6,7) -> (5,3,1,2,4,6)
$14$ & \texttt{35 332 1605 4786 9591 13476 13518 9681 4861 1640 341 36} & $(2, 1463725)$, $(3, 1453627)$, $(4, 1465723)$, $(c, 1365724): c \in \{5,6\}$ \\
$15$ & \texttt{35 332 1597 4714 9304 12811 12531 8708 4224 1373 276 29} & $(2, 3452617)$, $(3, 2456713)$, $(4, 1456723)$, $(c, 1356724): c \in \{5,6\}$\\
$16$ & \texttt{35 332 1607 4802 9647 13588 13658 9793 4917 1656 343 36} & None \\ % Does not contain a sub-weight of type EEEE
$17$ & \texttt{35 331 1596 4749 9499 13322 13336 9527 4769 1603 332 35} & None \\ % Does not contain a sub-weight of type EEEE
\hline
$18$ & \texttt{35 332 1603 4769 9527 13336 13322 9499 4749 1596 331 35} & $(2, 1364725)$, $(3, 1354627)$, $(c, 1367425): c \in \{4,5,6\}$ \\
$19$ & \texttt{35 332 1588 4642 9050 12293 11859 8134 3902 1259 253 27} & $(2, 1346725)$, $(3, 1345627)$, $(4, 2345671)$, $(5, 2456713)$, $(6, 1467532)$\\
$21$ & \texttt{35 331 1585 4652 9121 12468 12104 8351 4027 1305 263 28} & None \\ % Does not contain a sub-weight of type EEEE
$22$ & \texttt{35 331 1576 4579 8859 11922 11376 7707 3649 1163 232 25} & None \\ % Does not contain a sub-weight of type EEEE
$23$ & \texttt{35 331 1570 4532 8697 11600 10970 7371 3467 1101 220 24} & None \\ % Does not contain a sub-weight of type EEEE
\hline
$24$ & \texttt{35 331 1584 4645 9100 12433 12069 8330 4020 1304 263 28} & None \\ % Does not contain a sub-weight of type EEEE
$25$ & \texttt{35 329 1566 4573 8932 12181 11817 8162 3948 1286 261 28} & $(c, 1354267): c \in \{1,2,3\}$, $(c, 3541267): c \in \{4,5,6\}$ \\
$26$ & \texttt{35 329 1558 4503 8663 11585 10978 7384 3473 1102 220 24} & $(c, 1345267): c \in \{1,2,3\}$, $(c, 3451267): c \in \{4,5,6\}$ \\
$27$ & \texttt{35 329 1568 4588 8981 12272 11922 8239 3983 1295 262 28} & None \\ % Does not contain a sub-weight of type EEEE
$28$ & \texttt{35 329 1535 4329 8084 10474 9625 6301 2904 913 184 21} & $(c, 1342567): c \in \{1,2,3\}$, $(c, 1236745): c \in \{4,5,6\}$\\
\hline
$29$ & \texttt{35 329 1567 4581 8960 12237 11887 8218 3976 1294 262 28} & None \\ % Does not contain a sub-weight of type EEEE
$30$ & \texttt{35 329 1555 4483 8606 11495 10893 7336 3458 1100 220 24} & $(c, 1352467): c \in \{1,2,3\}$, $(c, 1263745): c \in \{4,5 \}$, $(6, 1263547)$\\
$31$ & \texttt{35 329 1565 4566 8911 12146 11782 8141 3941 1285 261 28} & $(c, 1354627): c \in \{1,2\}$, $(c, 1453267): c \in \{3,4,5,6\}$\\
$32$ & \texttt{35 329 1555 4483 8606 11495 10893 7336 3458 1100 220 24} & $(c, 1345627): c \in \{1,2\}$, $(3, 3452167)$, $(4, 1245637)$, $(c, 4532167): c \in \{5,6\}$\\
$33$ & \texttt{35 329 1565 4565 8903 12118 11726 8071 3885 1257 253 27} & $(3, 5672413)$, $(c, 1362745): c \in \{4,5,6\}$\\
\hline
$34$ & \texttt{35 329 1514 4177 7599 9579 8573 5485 2487 778 159 19} & $(c, 1234567): c \in [6]$ \\
$35$ & \texttt{35 329 1528 4276 7907 10132 9204 5959 2721 851 172 20} & $(c, 1235467): c \in [6]$ \\
$36$ & \texttt{35 329 1546 4411 8352 10977 10221 6762 3136 986 197 22} & None \\ % Does not contain a sub-weight of type EEEE
$37$ & \texttt{35 329 1548 4424 8388 11032 10271 6789 3144 987 197 22} & None \\ % Does not contain a sub-weight of type EEEE
$38$ & \texttt{35 329 1546 4411 8352 10977 10221 6762 3136 986 197 22} & $(c, 1253467): c \in [6]$ \\
\hline
$39$ & \texttt{35 329 1558 4503 8663 11585 10978 7384 3473 1102 220 24} & None \\ % Does not contain a sub-weight of type EEEE
$41$ & \texttt{35 329 1558 4503 8663 11585 10978 7384 3473 1102 220 24} & None \\ % Does not contain a sub-weight of type EEEE
$42$ & \texttt{35 329 1555 4483 8606 11495 10893 7336 3458 1100 220 24} & None \\ % Does not contain a sub-weight of type EEEE
$44$ & \texttt{35 329 1546 4411 8352 10977 10221 6762 3136 986 197 22} & None \\ % Does not contain a sub-weight of type EEEE
$45$ & \texttt{35 330 1586 4705 9387 13140 13140 9387 4705 1586 330 35} & $(5, 2564713)$, $(6, 5624713)$ \\
\hline
$46$ & \texttt{35 335 1638 4942 10011 14190 14316 10269 5137 1715 350 36} &  None -- has EEEE as a sub-weight \\ %does not arise from a matching field \\ % sub-weight of type EEEE remove 1, apply (2,3,4,5,6,7) -> (1,2,6,5,3,4)
$47$ & \texttt{35 333 1614 4822 9675 13602 13644 9765 4897 1649 342 36} & $(2, 2365714)$, $(3, 1365724)$, $(c, 1354627): c \in \{4,5,6\}$ \\
$48$ & \texttt{35 333 1607 4760 9431 13042 12818 8953 4365 1425 287 30} & $(2, 2356714)$, $(3, 1356724)$, $(c, 1345627): c \in \{4,5,6\}$ \\
$49$ & \texttt{35 333 1617 4846 9759 13770 13854 9933 4981 1673 345 36} & None \\ % Does not contain a sub-weight of type EEEE
$52$ & \texttt{35 331 1589 4686 9248 12741 12475 8680 4216 1372 276 29} & $(2, 2367514)$, $(3, 1367524)$, $(4, 1356427)$, $(c, 1456732): c \in \{5,6\}$ \\
\hline
$53$ & \texttt{35 332 1600 4739 9396 13007 12797 8946 4364 1425 287 30} & $(2, 3561724)$, $(3, 2367415)$, $(4, 1367524)$, $(c, 1356427): c \in \{5,6\}$\\
$54$ & \texttt{35 331 1591 4704 9319 12902 12706 8897 4349 1423 287 30} & $(3, 1462735)$, $(4, 1463752)$, $(c, 1453672): c \in \{5,6\}$ \\
$57$ & \texttt{35 331 1583 4637 9072 12377 11999 8274 3992 1296 262 28} & $(2, 2365174)$, $(3, 1365274)$, $(c, 1265734): c \in \{4,5,6\}$ \\
$58$ & \texttt{35 331 1577 4586 8880 11957 11411 7728 3656 1164 232 25} & $(2, 2356174)$, $(3, 1356274)$, $(c, 1256734): c \in \{4,5,6\}$ \\
$59$ & \texttt{35 331 1587 4667 9170 12559 12209 8428 4062 1314 264 28} & None \\ % Does not contain a sub-weight of type EEEE
\hline
$60$ & \texttt{35 331 1584 4645 9100 12433 12069 8330 4020 1304 263 28} & None \\ % Does not contain a sub-weight of type EEEE
$61$ & \texttt{35 331 1582 4630 9051 12342 11964 8253 3985 1295 262 28} & $(3, 1264735)$, $(c, 1254637): c \in \{4,5,6\}$ \\
$62$ & \texttt{35 331 1570 4532 8697 11600 10970 7371 3467 1101 220 24} & $(2, 3456127)$, $(3, 1246735)$, $(c, 1245637): c \in \{4,5,6\}$\\
$63$ & \texttt{35 329 1562 4537 8789 11852 11334 7693 3647 1163 232 25} & $(2, 3561274)$, $(3, 3451267)$, $(4, 1563274)$, $(c, 1267534): c \in \{5, 6 \}$\\
$64$ & \texttt{35 329 1562 4537 8789 11852 11334 7693 3647 1163 232 25} & $(c, 1356247): c \in \{1,2 \}$, $(3, 1452637)$, $(4, 1453672)$, $(c, 4536172): c \in \{5, 6\}$\\
\hline
$65$ & \texttt{35 329 1558 4503 8663 11585 10978 7384 3473 1102 220 24} & $(4, 1267534)$, $(5, 1367542)$, $(6, 1563274)$ \\ % Does not contain a sub-weight of type EEEE
$68$ & \texttt{35 329 1528 4276 7907 10132 9204 5959 2721 851 172 20} & $(c, 1235647): c \in [6]$ \\
$69$ & \texttt{35 329 1558 4503 8663 11585 10978 7384 3473 1102 220 24} & None \\ % Does not contain a sub-weight of type EEEE
$70$ & \texttt{35 329 1546 4411 8352 10977 10221 6762 3136 986 197 22} & $(c, 1253647): c \in [6]$\\
$71$ & \texttt{35 329 1554 4475 8578 11439 10823 7280 3430 1092 219 24} & $(c, 1254637): c \in \{1,2,3\}$, $(4, 1354672)$, $(c, 5613247): c \in \{5,6\}$\\
\hline
$72$ & \texttt{35 329 1546 4411 8352 10977 10221 6762 3136 986 197 22} & $(c, 1245637): c \in \{1,2,3\}$, $(4, 1345672)$, $(c, 4563271): c \in \{5, 6\}$\\
$73$ & \texttt{35 329 1555 4483 8606 11495 10893 7336 3458 1100 220 24} & $(c, 1352647): c \in [6]$ \\
$74$ & \texttt{35 329 1554 4475 8578 11439 10823 7280 3430 1092 219 24} & $(c, 1254367): c \in [6]$ \\
$75$ & \texttt{35 329 1548 4424 8388 11032 10271 6789 3144 987 197 22} & $(c, 1245367): c \in [6]$\\
$76$ & \texttt{35 329 1561 4530 8768 11817 11299 7672 3640 1162 232 25} & $(c, 1356427): c \in \{1,2,3\}$, $(4, 2356471)$, $(5, 1675423)$, $(6, 3674251)$\\
\hline
$77$ & \texttt{35 329 1558 4503 8663 11585 10978 7384 3473 1102 220 24} & None \\ % Does not contain a sub-weight of type EEEE
$78$ & \texttt{35 329 1528 4276 7907 10132 9204 5959 2721 851 172 20} & None \\ % Does not contain a sub-weight of type EEEE
$82$ & \texttt{35 329 1565 4566 8911 12146 11782 8141 3941 1285 261 28} & $(3, 2564173)$, $(4, 2367451)$, $(c, 1267435): c \in \{5,6\}$ \\
$83$ & \texttt{35 329 1546 4411 8352 10977 10221 6762 3136 986 197 22} & None \\ % Does not contain a sub-weight of type EEEE
$84$ & \texttt{35 329 1546 4411 8352 10977 10221 6762 3136 986 197 22} & None \\ % Does not contain a sub-weight of type EEEE
\hline
$89$ & \texttt{35 329 1546 4411 8352 10977 10221 6762 3136 986 197 22} & $(2, 4561237), (3, 2461753), (c, 1452673): c \in \{4, 5, 6\}$\\
$90$ & \texttt{35 329 1548 4424 8388 11032 10271 6789 3144 987 197 22} & $(c, 1256347): c \in \{1,2,3\}$, $(c, 1256473): c \in \{4,5,6\}$\\
$94$ & \texttt{35 329 1548 4424 8388 11032 10271 6789 3144 987 197 22} & $(c, 1256437): c \in [6]$\\
$98$ & \texttt{35 329 1554 4475 8578 11439 10823 7280 3430 1092 219 24} & $(3, 3564172)$, $(4, 1267435)$, $(c, 2453167): c \in \{5,6\}$\\
        \bottomrule
     \caption{Maximal prime cones of $\Trop(\Gr(3,7)^{\circ})$ and their associated matching fields  $B_c^\sigma$. % polytopes.
     } \label{tab: gr37 matching field polytopes}
    \end{longtable}
    \endgroup
% }
}
\end{example}

\begin{example}[The flag variety $\Fl_4$]
\label{example: compute fl4}
{In \cite{bossinger2017computing},} it is shown that $\Trop(\Fl_4^{\circ})$ has $78$ maximal cones, of which $72$ are prime. They are partitioned into 4 orbits. 
In \Cref{table: flag 4 matching fields} we list the prime cone orbits along with matching fields $B_c^\sigma$ which give rise to it.

\begin{example}[The flag variety $\Fl_5$] \label{example: compute fl5}
In \cite{bossinger2017computing}, it is shown that $\Trop(\Fl_5^{\circ})$ has $69780$ maximal cones which are partitioned into $536$ orbits under the action of $S_5 \times \ZZ_2$. There are $180$ prime cone orbits which give rise to non-isomorphic toric degenerations. In \Cref{table: flag 5 matching fields} we list the prime cones that correspond to toric degenerations via matching fields $B_c^\sigma$. The table follows the notation in \cite{bossinger2017computing}, giving the \textit{orbit index} and \textit{combinatorial equivalences} that include the string polytopes S$n$ and the Gelfand-Tsetlin polytope.

\begin{table}[ht]
    {\centering
    \begingroup\fontsize{7}{1em}\selectfont
    \begin{tabular}{lllll}
    \toprule
       Orbit  & Size & $F$-vector & $(c, \sigma)$ such that $B_c^\sigma$ induces the degeneration & Combinatorial equiv.\\
    \midrule
        $1$ & $24$ & \texttt{42 141 202 153 63 13} & $(c, 2341) : c \in \{ 3, 4\}$ & S2 \\
        $2$ & $12$ & \texttt{40 132 186 139 57 12} & $(c, 2341) : c \in \{1,2\}$, $(c, 3421) : c \in \{3,4\}$ & S1 (Gelfand-Tsetlin) \\
        $3$ & $12$ & \texttt{42 141 202 153 63 13} & $(c, 1342) : c \in [4]$ & S3 (FFLV)\\
        $4$ & $24$ & \texttt{43 146 212 163 68 14} & $(c, 1234) : c \in [4]$\\
    \bottomrule
    \end{tabular}
    \caption{Maximal prime cone orbits in $\Trop(\Fl_4^{\circ})$ and their associated matching fields $B_c^\sigma$.}
    \label{table: flag 4 matching fields}
    \endgroup}
\end{table}
\end{example}

{
    \centering
    % \resizebox{\textwidth}{!}{
    \begingroup\fontsize{6.2}{1em}\selectfont
    \begin{longtable}[ht]{llll}
        \toprule
        Orbit & Polytope f-vector & $(c, \sigma)$ s.t. $B_c^\sigma$ induces the degeneration & Combinatorial equiv.\\
        \midrule 
        \endhead
        $32$ & \texttt{397 2363 6416 10313 10755 7536 3561 1109 215 23} &  $(c, 14532) : c \in [6] $ & S22\\
%        $33$ & \texttt{425 2553 6988 11317 11888 8388 3987 1245 240 25} & &\\
%        $34$ & \texttt{419 2522 6922 11243 11842 8373 3985 1245 240 25} & &\\
       % \hline
%        $35$ & \texttt{405 2407 6518 10442 10851 7578 3571 1110 215 23} & &\\
%        $36$ & \texttt{401 2387 6477 10398 10825 7570 3570 1110 215 23} & &\\
        $37$ & \texttt{368 2154 5755 9111 9373 6497 3052 953 188 21} & $(c, 24531) : c \in \{1,2,3\} $ & S21\\
        $38$ & \texttt{379 2214 5892 9280 9494 6547 3063 954 188 21} & $(c, 45312) : c \in \{1,2,3\}$ & S27, S28\\
%        $39$ & \texttt{393 2313 6200 9833 10125 7021 3297 1027 201 22} & &\\
       % \hline
        $40$ & \texttt{358 2069 5453 8516 8653 5941 2778 870 174 20} & $(c, 45321) : c \in [6]$ & S1, S18, S26,
        S29 (Gelfand-Tsetlin)\\
        $65$ & \texttt{428 2608 7269 12028 12946 9377 4576 1462 285 29} & $(c, 12435) : c \in [6]$ &\\
        \midrule
   %     $66$ & \texttt{441 2681 7438 12228 13056 9369 4525 1430 276 28} & &\\
        $67$ & \texttt{418 2510 6876 11157 11753 8321 3969 1243 240 25} & $(c, 12345) : c \in [6]$ &\\
        $68$ & \texttt{406 2442 6713 10943 11587 8245 3950 1241 240 25} & $(c, 13452) : c \in \{1,2,3\}$, $(c, 34512) : c \in \{4,5,6\}$ 
        %&\\ && $(c, 34512) : c \in \{4,5,6\}$ 
        &\\
        $69$ & \texttt{373 2199 5926 9474 9849 6897 3267 1024 201 22} & $(c, 23451) : c \in \{1,2\}$ &\\
   %     \hline
   %     $70$ & \texttt{427 2586 7144 11681 12383 8806 4209 1317 253 26} & &\\
   %     $71$ & \texttt{451 2781 7840 13111 14243 10390 5089 1623 313 31} & &\\
   %     $72$ & \texttt{440 2704 7602 12684 13752 10014 4897 1560 301 30} & &\\
   % %    $73$ & \texttt{406 2442 6713 10943 11587 8245 3950 1241 240 25} & &\\
   %     $74$ & \texttt{448 2764 7800 13061 14208 10377 5087 1623 313 31} & &\\
   %     \hline
   %     $75$ & \texttt{462 2873 8181 13846 15258 11321 5656 1844 363 36} & &\\
   %     $76$ & \texttt{457 2842 8099 13726 15153 11266 5640 1842 363 36} & &\\
   %     $77$ & \texttt{469 2927 8364 14203 15699 11678 5845 1907 375 37} & &\\
   %     $78$ & \texttt{454 2802 7903 13216 14348 10453 5110 1626 313 31} & &\\
   %     $79$ & \texttt{451 2787 7879 13221 14419 10565 5200 1667 323 32} & &\\
   %     \hline
   %     $80$ & \texttt{441 2705 7584 12611 13622 9885 4823 1537 298 30} & &\\
   %     $81$ & \texttt{454 2803 7914 13263 14455 10598 5231 1687 330 33} & &\\
   %     $82$ & \texttt{441 2697 7532 12465 13391 9660 4685 1485 287 29} & &\\
   %     $83$ & \texttt{445 2721 7593 12550 13461 9694 4694 1486 287 29} & &\\
   %     $84$ & \texttt{441 2697 7532 12465 13391 9660 4685 1485 287 29} & &\\
   %     \hline
    %    $85$ & \texttt{445 2725 7617 12611 13546 9764 4728 1495 288 29} & &\\
        $86$ & \texttt{397 2363 6416 10313 10755 7536 3561 1109 215 23} & $(3, 34512)$ &\\
        $87$ & \texttt{368 2154 5755 9111 9373 6497 3052 953 188 21} & $(4, 34521)$ & S5, S31\\
        \midrule
    %    $88$ & \texttt{452 2801 7946 13385 14654 10771 5309 1699 327 32} & &\\
    %    $89$ & \texttt{430 2624 7318 12097 12974 9329 4497 1411 269 27} & &\\
    %    \hline
    %    $90$ & \texttt{456 2834 8071 13670 15083 11210 5612 1834 362 36} & &\\
    %    $91$ & \texttt{432 2633 7332 12104 12975 9341 4521 1430 276 28} & &\\
    %    $92$ & \texttt{467 2919 8359 14230 15769 11756 5892 1922 377 37} & &\\
    %    $93$ & \texttt{456 2834 8071 13670 15083 11210 5612 1834 362 36} & &\\
        $94$ & \texttt{426 2597 7244 11998 12926 9370 4575 1462 285 29} & $(c, 13425) : c \in [6]$ &\\
      %  \hline
        $95$ & \texttt{440 2708 7630 12769 13898 10169 5001 1603 311 31} & $(c, 13452) : c \in \{4,5,6\}$ &\\
    %    $96$ & \texttt{432 2633 7332 12104 12975 9341 4521 1430 276 28} & &\\
        $97$ & \texttt{412 2479 6810 11083 11707 8306 3967 1243 240 25} & $(c, 34215) : c \in [6]$ &\\
        $98$ & \texttt{415 2511 6945 11391 12133 8679 4174 1313 253 26} & $(3, 23451)$ &\\
    %    $99$ & \texttt{458 2845 8092 13676 15042 11132 5543 1800 353 35} & &\\
    %    \hline
    %    $100$ & \texttt{437 2669 7447 12319 13236 9556 4642 1475 286 29} & &\\
    %    $101$ & \texttt{441 2703 7569 12562 13531 9780 4746 1502 289 29} & &\\
    %    $102$ & \texttt{427 2586 7144 11681 12383 8806 4209 1317 253 26} & &\\
    %    $103$ & \texttt{419 2522 6922 11243 11842 8373 3985 1245 240 25} & &\\
    %    $104$ & \texttt{437 2669 7447 12319 13236 9556 4642 1475 286 29} & &\\
    %    \hline
    %    $105$ & \texttt{411 2470 6776 11012 11617 8235 3933 1234 239 25} & &\\
     %   $106$ & \texttt{413 2483 6808 11043 11606 8177 3871 1201 230 24} & &\\
    %    $107$ & \texttt{425 2553 6988 11317 11888 8388 3987 1245 240 25} & &\\
    %    $108$ & \texttt{405 2407 6518 10442 10851 7578 3571 1110 215 23} & &\\
%        $109$ & \texttt{405 2427 6638 10751 11296 7969 3785 1181 228 24} & & S30 \\
    %    \hline
    %    $110$ & \texttt{465 2904 8312 14152 15700 11734 5907 1940 384 38} & &\\
    %    $111$ & \texttt{464 2902 8323 14204 15795 11828 5960 1956 386 38} & &\\
    %    $112$ & \texttt{438 2690 7559 12608 13667 9952 4868 1552 300 30} & &\\
    %    $113$ & \texttt{445 2725 7617 12611 13546 9764 4728 1495 288 29} & &\\
    %    $114$ & \texttt{437 2669 7447 12319 13236 9556 4642 1475 286 29} & &\\
    %    \hline
        $115$ & \texttt{411 2470 6776 11012 11617 8235 3933 1234 239 25} & $(c, 45123) : c \in [6]$ &\\
        \midrule
    %    $116$ & \texttt{424 2574 7139 11737 12529 8983 4332 1367 264 27} & &\\
    %    $117$ & \texttt{419 2522 6922 11243 11842 8373 3985 1245 240 25} & &\\
        $118$ & \texttt{401 2387 6477 10398 10825 7570 3570 1110 215 23} & $(c, 45213) : c \in \{1,2,3,4\}$ &\\
        $119$ & \texttt{405 2427 6638 10751 11296 7969 3785 1181 228 24} & $(c, 34521) : c \in \{5,6\}$ & S6\\
    %    \hline
    %    $120$ & \texttt{464 2893 8261 14019 15483 11503 5746 1869 366 36} & &\\
    %    $121$ & \texttt{454 2806 7928 13283 14448 10543 5159 1641 315 31} & &\\
    %    $122$ & \texttt{451 2794 7928 13370 14676 10840 5387 1746 342 34} & &\\
    %    $123$ & \texttt{444 2736 7715 12915 14053 10273 5044 1613 312 31} & &\\
    %    $124$ & \texttt{466 2909 8318 14138 15644 11650 5837 1906 375 37} & &\\
   %     \hline
    %    $125$ & \texttt{456 2815 7939 13271 14398 10480 5118 1627 313 31} & &\\
        $126$ & \texttt{423 2561 7078 11586 12303 8767 4199 1316 253 26} & $(c, 23415) : c \in \{4,5,6\}$ &\\
     %   $127$ & \texttt{429 2580 7064 11429 11972 8402 3959 1221 232 24} & &\\
     %   $128$ & \texttt{431 2626 7309 12058 12915 9290 4494 1422 275 28} & &\\
     %   $129$ & \texttt{428 2602 7224 11883 12684 9087 4375 1377 265 27} & &\\
   %     \hline
    %    $130$ & \texttt{443 2727 7679 12831 13927 10147 4960 1577 303 30} & &\\
    %%    $131$ & \texttt{432 2637 7354 12152 13024 9356 4505 1412 269 27} & &\\
    %    $132$ & \texttt{451 2793 7920 13342 14620 10770 5331 1718 334 33} & &\\
    %    $133$ & \texttt{434 2632 7273 11879 12557 8883 4210 1301 246 25} & &\\
    %    $134$ & \texttt{452 2781 7813 13004 14042 10171 4944 1566 301 30} & &\\
     %   \hline
     %   $135$ & \texttt{453 2808 7969 13433 14725 10847 5366 1727 335 33} & &\\
    %    $136$ & \texttt{451 2794 7928 13370 14676 10840 5387 1746 342 34} & &\\
        $137$ & \texttt{433 2646 7390 12236 13150 9482 4589 1448 278 28} & $(c, 23451) : c \in \{4,5,6\}$ &\\
    %    $138$ & \texttt{442 2715 7629 12727 13808 10076 4948 1587 309 31} & &\\
    %    $139$ & \texttt{432 2633 7332 12104 12975 9341 4521 1430 276 28} & &\\
    %    \hline
    %    $140$ & \texttt{423 2564 7096 11632 12368 8822 4227 1324 254 26} & &\\
    %%    $141$ & \texttt{413 2483 6808 11043 11606 8177 3871 1201 230 24} & &\\
    %    $142$ & \texttt{427 2594 7196 11827 12614 9031 4347 1369 264 27} & &\\
    %    $143$ & \texttt{431 2622 7281 11973 12769 9135 4390 1379 265 27} & &\\
    %    $144$ & \texttt{431 2626 7309 12058 12915 9290 4494 1422 275 28} & &\\
   %     \hline
        $145$ & \texttt{410 2459 6725 10881 11411 8029 3802 1183 228 24} & $(c, 45231) : c \in \{4,5,6\}$ &\\
        \midrule
    %    $146$ & \texttt{428 2594 7176 11761 12514 8947 4307 1359 263 27} & &\\
        $147$ & \texttt{419 2522 6922 11243 11842 8373 3985 1245 240 25} & $(c, 24531) : c \in \{5,6\}$ &\\
    %    $148$ & \texttt{451 2781 7840 13111 14243 10390 5089 1623 313 31} & &\\
    %    $149$ & \texttt{464 2900 8310 14168 15740 11778 5933 1948 385 38} & &\\
    %    \hline
    %    $150$ & \texttt{446 2750 7757 12985 14123 10315 5058 1615 312 31} & &\\
        $151$ & \texttt{420 2541 7021 11496 12218 8719 4184 1314 253 26} & $(c, 12453) : c \in [6]$ &\\
    \bottomrule
    \caption{Maximal prime cones of $\Trop(\Fl_5^{\circ})$ and their associated matching fields  $B_c^\sigma$.}
    \label{table: flag 5 matching fields}
    \end{longtable}
    \endgroup
% }
}
\end{example}

\bibliographystyle{ieeetr}
\bibliography{bibfile.bib}

\noindent 
\textbf{Author's addresses}

\noindent
School of Mathematics, University of Edinburgh, Edinburgh, United Kingdom
\\ E-mail address: {\tt oliver.clarke@ed.ac.uk}

\medskip  \noindent
Department of Computer Science, KU Leuven, Celestijnenlaan 200A, B-3001 Leuven, Belgium\\ 
   Department of Mathematics, KU Leuven, Celestijnenlaan 200B, B-3001 Leuven, Belgium\\
 Department of Mathematics and Statistics,
 UiT – The Arctic University of Norway, 9037 Troms\o, Norway
 \\ E-mail address: {\tt fatemeh.mohammadi@kuleuven.be}

\medskip  \noindent
Department of Mathematics, KU Leuven, Celestijnenlaan 200B, B-3001 Leuven, Belgium
\\ E-mail address: {\tt francesca.zaffalon@kuleuven.be}

\end{document}